\documentclass[12 pt]{article}        	
\usepackage{amsfonts, amssymb}					 
\usepackage[colorlinks=true,citecolor=teal,linkcolor=blue]{hyperref}
\usepackage{tikz}
\usepackage{tikz-cd}
\usepackage{quiver} 
\usepackage{xcolor}
\usepackage{graphicx} 
\usepackage{amsmath}
\usepackage{mathrsfs}
\usepackage{mathabx}
\usepackage{comment}
\usepackage{setspace} 
\usepackage{enumerate}
\usepackage{amsthm}
\usepackage{multicol}

\usepackage{url}

\newcommand{\eqn}[0]{\begin{array}{rcl}}
\newcommand{\eqnend}[0]{\end{array} }  	
\newcommand\norm[1]{\left\lVert#1\right\rVert}
\newcommand\restr[2]{{
  \left.\kern-\nulldelimiterspace 
  #1
  \vphantom{\big|} 
  \right|_{#2} 
  }}
  \newcommand{\inn}[1]{\langle#1\rangle}
\def\P{\mathbb{P}}
\def\E{\mathbb{E}}

\def\R{\mathbb{R}}
\numberwithin{equation}{section}
\newtheorem{theorem}{Theorem}[section]
\newtheorem{prop}[theorem]{Proposition}

\theoremstyle{definition}
\newtheorem{definition}[theorem]{Definition}

\theoremstyle{definition}
\newtheorem{remark}[theorem]{Remark}

\theoremstyle{definition}
\newtheorem{problem}[theorem]{Problem}
\newtheorem{conjecture}[theorem]{Conjecture}

\theoremstyle{definition}

\newtheorem{corollary}[theorem]{Corollary}

\newtheorem{lemma}[theorem]{Lemma}
\theoremstyle{definition}
\newtheorem{notation}[theorem]{Notation}
\begin{document}	

\title{Optimal Covariance Estimates for Schr\"odinger Semigroups with White Noise in $d=1,2$}
\author{Youssef Djellouli\footnote{Syracuse University, Department of Mathematics. Email:
\href{mailto:mydjello@syr.edu}{mydjello@syr.edu}.} \and Pierre Yves Gaudreau Lamarre\footnote{Syracuse University, Department of Mathematics. Email: \href{mailto:pgaudrea@syr.edu}{pgaudrea@syr.edu}.}}
\date{}
\maketitle

\begin{abstract}
For $d\in\{1,2\}$, let $H=-\frac{1}{2}\Delta + V +\xi$ be the random Schr\"odinger operator on $L^2(\R^d)$ where $\xi$ is a standard Gaussian white noise
and $V$ is a deterministic potential with power-law growth at infinity. Using a Feynman-Kac formula for the trace of the Schr\"odinger semigroup, we give optimal asymptotic upper and lower bounds on the covariance of $\mathrm{Tr}[e^{-sH}]$ and $\mathrm{Tr}[e^{-tH}]$ as $s,t\to0$ through estimates on Brownian bridge local times. These estimates are a significant improvement on previous bounds in the case $d=1$ and are the first of their kind for $d=2$. As an application of these new estimates, we prove a quantitative hyperuniformity-type property and decorrelation rate for the trace as $s,t\to0$.
\end{abstract}

\section{Introduction}

\subsection{Schr\"odinger Operators with White Noise}\label{section: white noise}

For $d\geq1$, let $V:\mathbb R^d\to\mathbb R$ be a deterministic function with at least power-law growth at infinity (see \eqref{Equation: V Lower Bound}),
and let $\xi$ be a standard Gaussian white noise.
Informally, $\xi$ can be viewed as a centered Gaussian process
on $\mathbb R^d$ whose covariance function is
the Dirac delta distribution.
In this paper, we are interested in the statistical properties of the eigenvalues of the random Schrödinger operator
\begin{align}
\label{Equation: H}
H=-\tfrac12\Delta+V+\xi,
\end{align}
which acts on sufficiently regular functions
in $L^2(\mathbb R^d)$. In physical terms, this
consists of studying the energy levels of an
isolated quantum particle that is subjected
to the confining potential $V$ and impurities/disorder modelled by $\xi$.

The motivation for using a white noise perturbation in \eqref{Equation: H} can in part
be explained by the fact that $\xi$ is the universal scaling limit of i.i.d. lattice noises with finite variance  (as the spacing between lattice points vanishes). Thus, $\xi$ is arguably one of the most natural models of unstructured noise on $\mathbb R^d$, and can be viewed as a continuum analogue of Anderson's celebrated tight-binding model \cite{Anderson}.
However, despite this natural interest,
most rigorous results concerning these models are quite recent, especially for $d\geq2$.
Much of the reason for this has to do with
the fact that Gaussian white noise is highly singular (in fact, a Schwartz distribution), which makes the rigorous definition and study of operators of the form
\eqref{Equation: H} technically challenging.

More specifically, for $d=1$, the rigorous theory of Schr\"odinger operators with white noise was pioneered by Fukushima and Nakao \cite{FukushimaNakao}. Therein, they showed how the operator $-\tfrac12\Delta+\xi$ on the bounded interval $[0,L]$ (with Dirichlet boundary conditions) can be constructed through its quadratic form, wherein integrals against $\xi$ are interpreted as stochastic integrals. They then used their construction to study the integrated density of states of the operator on large intervals, thus confirming a physics prediction due to Frisch, Halperin, and Lloyd \cite{FrischLloyd,Halperin}. Over the past 30 years, the spectral properties of that same operator (such as eigenvalue tails and localization-to-delocalization transition) were studied in great detail in a series of works by various authors
\cite{CambroneroMcKean,CambroneroRider,McKean,DumazLabbe,DumazLabbe2,DumazLabbe3,DumazLabbe4}.
Finally, more general one-dimensional Schr\"odinger operators (e.g., with nonzero $V$) with white noise were considered by various authors, motivated in part by applications in multivariate statistics and the study of simple
interacting particle systems (namely, the $\beta$-ensembles); e.g., \cite{BloemendalVirag1,BloemendalVirag2,DumazLabbe5,GLFK,LGL,AiryRigidity,Minami,Minami2,RamirezRiderVirag}.

For $d\geq2$, the rigorous construction of Schr\"odinger operators with white noise is substantially more difficult, and typically requires sophisticated renormalization theories, such as Martin Hairer's regularity structures \cite{Hairer}
or the paracontrolled calculus of
Gubinelli, Imkeller, and Perkowski
\cite{GIP}. Using this technology, Allez and Chouk constructed the operator $-\tfrac12\Delta+\xi$ on the two-dimensional torus \cite{AllezChouk}, and studied the tails of its eigenvalues. Then, Labb\'e constructed the same operator for $d\leq3$ on the
torus or a box with Dirichlet boundary conditions \cite{Labbe}, and also studied
the eigenvalue tails. Since then,
other properties of $-\tfrac12\Delta+\xi$'s spectrum on various spaces and manifolds of dimension $d\leq3$ were investigated, such as exact eigenvalue asymptotics, density of states, and spectral geometry; e.g.,
\cite{BailleulDangMouzard,ChoukVanZuijlen,GLP,HsuLabbe,Matsuda,MatsudaVanZuijlen,Mouzard}.

\begin{remark}
To our knowledge, no Schr\"odinger operators have been constructed with white noise in $d\geq4$. As explained in several papers
(e.g. \cite[(PAM) and surrounding paragraph]{Hairer} or \cite[Paragraph between (2) and (3)]{ChoukGairingPerkowski}), this has to do with the balance between the scaling/regularity properties of the Laplacian and the white noise, whereby the noise
``overpowers" the Laplacian when $d\geq4$ even if an additive renormalization is applied. This is consistent with two
recent works concerning the $d=4$ case \cite{DengShen,GabrielRosati}, which show that
a logarithmic attenuation of the noise leads to
nontrivial fluctuations in $-\tfrac12\Delta+\xi$'s Green
function on a torus.
\end{remark}

\subsection{Main Result}
\label{sec: Main Result}

In this paper, our aim is to study the fluctuations and correlations of $H$'s
large eigenvalues. The tool that we use
for this purpose is the semigroup covariance
\[C(s,t)=\mathrm{Cov}\left[\mathrm{Tr}[e^{-s H}],\mathrm{Tr}[e^{-t H}]\right],\qquad s,t>0,\]
which we analyze using the Feynman-Kac formula.
In order to ensure that this object is well-defined
and amenable to computation,
throughout the paper we assume the following:
\begin{enumerate}
\item $d=1,2$. This is because for $d=3$, it is known that $\mathrm{Tr}[e^{-tH}]$'s moments cannot be finite (see, e.g., \cite[Theorem 2 and the following paragraph]{Labbe}), thus rendering the quantity $C(s,t)$ meaningless.
\item There exist constants $\kappa,a,b>0$ such that
\begin{align}
\label{Equation: V Lower Bound}
V(x)\geq a|x|^\kappa-b,\qquad x\in\mathbb R^d.
\end{align}
This is meant to ensure that $H$ has a purely discrete spectrum that is bounded below, and that $e^{-tH}$ is trace class.
\item $V_+=\max\{0,V\}\in K_d^{\mathrm{loc}}$,
where $K_d^{\mathrm{loc}}$ denotes the local
Kato class (see, e.g., \cite[Section A2]{Simon} for a definition). This is meant to ensure that the Feynman-Kac formula with potential $V$ is well-defined.
\end{enumerate}

\begin{remark}
When $d=1$, the construction of the operator $H$ under the three assumptions above has been carried out in \cite{GLFK}.
When $d=2$, to the best of our knowledge, the construction
of $H$ in this general setting has not been carried out yet.
Following \cite{AllezChouk,Labbe}, we expect that the latter would require regularity
structures or paracontrolled calculus. However, given that we are only interested in the trace $\mathrm{Tr}[e^{-tH}]$ in this paper, we are able to circumvent these
difficulties entirely and construct the trace observable for a fixed $t>0$ as a simple $L^2$ limit
(in the sense of square-integrable
real-valued random variables).
See Section \ref{Section: Trace Set Up} for the details
of $H$'s construction when $d=1$, and $\mathrm{Tr}[e^{-tH}]$'s construction when $d=2$.
\end{remark}

With this in hand, our main result is as follows:

\begin{definition}
We say that $f(\tau)\lesssim g(\tau)$ as $\tau\to\tau_0$ if there exists a constant $c>0$ and a neighborhood $U$ of $\tau_0$ such that $f(\tau)\leq c g(\tau)$ for all $\tau\in U$.
We say that $f(\tau)\asymp g(\tau)$ as $\tau\to\tau_0$ if $f(\tau)\lesssim g(\tau)$ and $g(\tau)\lesssim f(\tau)$ as $\tau\to\tau_0$.
\end{definition}

\begin{theorem}\label{Thm: Main Theorem}
It holds that
\begin{align}
\label{eq: Main Upper Bound}
C(s,t)\lesssim\min\{s,t\}^{1-d/2}\max\{s,t\}^{1-d/2-d/\kappa}\qquad\text{ as }(s,t)\to0.
\end{align}
This result is optimal,
in the sense that if there also exist 
constants $c,f>0$ such that
\begin{align}
\label{Equation: V Upper Bound}
V(x)\leq c|x|^\kappa+f,
\qquad x\in\mathbb R^d,
\end{align}
then
\[C(s,t)\asymp\min\{s,t\}^{1-d/2}\max\{s,t\}^{1-d/2-d/\kappa}\qquad\text{as }(s,t)\to0.\]
\end{theorem}

See Sections \ref{upper bound proof} and \ref{optimality proof} for the proofs of the upper bound and the optimality respectively.

\subsection{Past Results and Motivation}

There have been a few past results that studied upper bounds on $C(s,t)$, mostly focused on the one-dimensional regime. More specifically, for $d=1$,
it was shown in \cite[Theorem 4.1]{LGL} that
\[\mathrm{Var}\big[\mathrm{Tr}[e^{-t H}]\big]=C(t,t)\lesssim t^{1-1/\kappa}\qquad\text{as }t\to0.\]
This matches the upper bound in Theorem \ref{Thm: Main Theorem} in the special case where $d=1$ and $s=t$.
A few years later, in \cite[Proposition 8.3]{AiryRigidity2} (see also \cite[Section 3.3]{AiryRigidity}),
it was shown that if $d=1$ and $\kappa=1$, then
\[C(s,t)\lesssim\left(\frac{\min\{s,t\}}{\max\{s,t\}}\right)^{1/4}\qquad\text{as }(s,t)\to0.\]
This is substantially weaker than the corresponding upper bound in \eqref{eq: Main Upper Bound}
(namely, with the power of $1/2$ replaced by $1/4$).

\begin{remark}
Although this is not made explicit in \cite{AiryRigidity2},
the method of proof used in \cite[Proposition 8.3]{AiryRigidity2} could be extended to $d=1$ and arbitrary $\kappa>0$ and yield the more general upper bound
\[C(s,t)\lesssim\min\{s,t\}^{1/4}\max\{s,t\}^{3/4-1/\kappa}\qquad\text{ as }(s,t)\to0.\]
\end{remark}

The motivation for obtaining these bounds in
\cite{LGL,AiryRigidity,AiryRigidity2} was to
establish a property called number rigidity
(see \cite{GhoshPeres}). For this purpose,
it is sufficient to either prove that $C(t,t)$
vanishes as $t\to0$, or that $C(s,t)$ vanishes
as $s\to0$ for $t>0$ fixed while $C(t,t)$ remains
bounded for small $t$. In particular, for such an application,
an optimal control on $C(s,t)$ is not necessary.

Finally, for $d=2$, it was proved in \cite{GLP}
that if $d=2$ and we consider $H=-\tfrac12\Delta+\xi$ restricted to a bounded domain with Dirichlet boundary---which formally corresponds to $\kappa=\infty$---one has $C(t,t)=O(1)$ as $t\to0$.
This again matches the two-dimensional upper bound in \eqref{eq: Main Upper Bound} if we formally interpret $2/\infty=0$. The motivation for that
result was to study the spectral geometry problem for $-\tfrac12\Delta+\xi$, whereby any feature of
a domain's heat trace that is larger than $O(1)$
as $t\to0$ can still be recovered from $H$'s trace
almost surely.

In this context, from the technical point of view, the main improvements on the past literature in this paper are as follows:
\begin{enumerate}
\item For $d=1$, we significantly improve
the control on $C(s,t)$ for $s\neq t$.
\item For $d=2$, we obtain, to our knowledge, the first
estimates on $C(s,t)$ for general potentials with power-law growth at infinity (i.e., \eqref{Equation: V Lower Bound}).
\item In all cases (i.e., $d=1,2$), we prove that \eqref{eq: Main Upper Bound} is the optimal upper bound on $C(s,t)$ under the growth assumption \eqref{Equation: V Lower Bound}
(via the matching lower bound when \eqref{Equation: V Upper Bound} holds).
\end{enumerate}
Our motivation for obtaining these improvements
comes from the new applications that are made
available once we obtain matching upper and lower
bounds on $C(s,t)$. Namely, as we explain in 
more detail in Section \ref{Section: Applications}, our main result makes it possible
to prove a type of ``semigroup hyperuniformity property" for $H$
(see Definition \ref{def:Hyperuniform in the sense of semigroups}), and to quantify precisely
the asymptotic correlations in $H$'s spectrum
as measured through the semigroup.

\subsection{Applications}
\label{Section: Applications}
In this section, we outline two applications of our results.

\subsubsection{Hyperuniformity in the Sense of Semigroups}

Given a point process $\mathcal X=\{x_k\}_{k\geq1}$ in $\mathbb R^n$ for some $n\geq1$, define the ball counting function
\[N_{\mathrm b}(r)=|\{k\geq1:|x_k|\leq r\}|,\]
and the corresponding
hyperuniformity ratio function
\[R_{\mathrm b}(r)=\frac{\mathrm{Var}[N_{\mathrm b}(r)]}{r^n},\qquad r>0.\]

\begin{definition}\label{Def: Hyperuniformity}
$\mathcal X$ is hyperuniform if $R_{\mathrm b}(r)=o(1)$ as $r\to\infty$.
\end{definition}

\begin{remark}
There is also a widely known characterization
of the strength of hyperuniformity (i.e., Class I, II, and III) due to Torquato \cite{HyperuniformSurvey}, which depends on
the rate at which $R_{\mathrm b}(r)$ vanishes.
\end{remark}

In words, a point process is called hyperuniform if the variance of the number of points in a ball grows more slowly than the ball's volume. This property is in sharp contrast with
the fluctuations that occur in homogeneous Poisson processes, which are such that \[\mathrm{Var}[N_{\mathrm b}(r)]\asymp\mathbb E[N_{\mathrm b}(r)]\asymp r^n\qquad\text{as } r\to\infty.\]

Informally, hyperuniformity is a ``global rigidity" property which means that despite the local disorder, at sufficiently large scales, the process behaves as a lattice.
In particular, hyperuniformity is closely related to several other notions, such as number rigidity or eigenvalue rigidity.
Hyperuniform processes are of particular interest to physicists and materials scientists due to applications arising from advantageous properties that come from being ``between" disordered liquid and ordered crystal states; see \cite{HyperuniformSurvey} for a survey.

Hyperuniformity and other rigidity phenomena have been studied for several decades in both physics (see e.g.,  \cite{1DCoulombGas,ReiszGasRigidityHyperuniformity,HyperuniformityOrderMetrics-Physics,WeylHeisenbergHyperuniformity,GinibreWeylHeisenbergHyperuniformity}) and mathematics (see e.g., \cite{NumberRigiditySuperhomogeneousFields, CoxPointProcesses,GibbsNon-Hyperuniform,PlasmaRigidity,StealthyHyperuniformity}).  Hyperuniformity and rigidity have been studied for point processes including Coulomb gases (e.g., \cite{CoulombChargeFluctuations,1DCoulombGas,Chatterjee,CoulombGasHyperuniformityGroundstates,CoulombLocalLawRigidity,ReiszGasRigidityHyperuniformity,StructureFactor,Leble}), eigenvalue processes of random unitary (e.g., \cite{UnitaryRandomMatrix}), Hermitian (e.g., \cite{HermitianRandomMatrix,Wigner}), and non-Hermitian (e.g. \cite{Ginibre,GinibreWeylHeisenbergHyperuniformity}) matrices, zeros of random analytic functions (e.g., \cite{zeros}), as well as many other processes (see e.g., \cite{CoxPointProcesses,NumberRigiditySuperhomogeneousFields,PlasmaRigidity,GibbsNon-Hyperuniform}).

Turning back to the topic of this paper,
let $\mathcal L=\sum_k\delta_{\lambda_k}$
denote $H$'s eigenvalue point process.
Given that $\mathcal L$ is not stationary and has a varying intensity over $\mathbb R$ (indeed, $\mathcal L$ is bounded below and its points grow at a varying rate, the latter of which should be determined by the Weyl law of the corresponding deterministic operator $-\frac12\Delta+V$), the traditional definition of
hyperuniformity is arguably not well-suited to
study its fluctuations.
A more natural analogue to Definition \ref{Def: Hyperuniformity} would be the vanishing of the half-line hyperuniformity ratio function
\begin{align}
\label{eqn: HL Ratio}
R_{\mathrm h}(\lambda)=\frac{\mathrm{Var}[N(\lambda)]}{\mathbb E[N(\lambda)]},\qquad \lambda\in\mathbb R
\end{align}
as $\lambda\to\infty$, where
\[N(\lambda)=|\{k\geq1:\lambda_k\leq\lambda\}|,\qquad\lambda\in\mathbb R\]
denotes $H$'s eigenvalue counting function. Indeed, such a property would imply
that $H$'s eigenvalue fluctuations are smaller than that of
a nonhomogeneous Poisson process whose intensity grows at the same rate as $\mathbb E[N(\lambda)]$ as $\lambda\to\infty$.

While Theorem \ref{Thm: Main Theorem} cannot be used to prove the vanishing of $R_{\mathrm h}(\lambda)$ directly,
it can nevertheless be used to establish yet another analogous property formulated using $H$'s semigroup.
More specifically, the classical Abelian/Tauberian theorems (e.g., \cite[\S XIII.5]{Feller}) relating cumulative distributions to Laplace transforms can be informally interpreted in the present context as the claim that
\begin{align}\label{equation: trace count relation}
\mathrm{Tr}[e^{-tH}]\approx N(1/t)\qquad\text{as } t\to0.
\end{align}
This leads us to the consideration of the semigroup hyperuniformity ratio function
\[R_{\mathrm s}(t)=\frac{\mathrm{Var}\left[\mathrm{Tr}[e^{-t H}]\right]}{\mathbb E\left[\mathrm{Tr}[e^{-t H}]\right]}=\frac{C(t,t)}{\mathbb E\left[\mathrm{Tr}[e^{-t H}]\right]},\qquad t>0,\]
and posit the following definition:

\begin{definition}
\label{def:Hyperuniform in the sense of semigroups}
The eigenvalue point process $\mathcal L$ is hyperuniform in the sense of semigroups if $R_{\mathrm s}(t)\to0$ as $t\to0$. 
\end{definition}

As a first step toward establishing the hyperuniformity of $\mathcal L$, we have the following:

\begin{prop}\label{Prop: Expectation Bound}
If \eqref{Equation: V Lower Bound} and \eqref{Equation: V Upper Bound} both hold, then
\[\mathbb E\left[\mathrm{Tr}[e^{-t H}]\right]\asymp t^{-d/2-d/\kappa}
\qquad\text{as }t\to0.\]
\end{prop}

See Section \ref{expectation bound proof} for proof. With this in hand, we can then obtain the following as an immediate corollary of 
Theorem \ref{Thm: Main Theorem}:

\begin{corollary}\label{cor: hyperuniformity ratio}
If \eqref{Equation: V Lower Bound} and \eqref{Equation: V Upper Bound} both hold, then
\[R_{\mathrm s}(t)\asymp t^{2-d/2}
\qquad\text{as }t\to0.\]
\end{corollary}

Thus, we not only obtain a semigroup hyperuniformity property, but we can also quantify the vanishing rate of the hyperuniformity ratio and identify that the latter only depends on the dimension, not the potential's growth.

\begin{remark}
Let $\{p_k\}_{k\geq1}$ denote a Poisson process on $[0,\infty)$
with intensity measure $\nu$ such that $\nu([0,\lambda])\sim \lambda^{d/2+d/\kappa}$ as $\lambda\to\infty$. Then, elementary calculations reveal that
\[\mathbb E\left[\sum_{k=1}^\infty e^{-tp_k}\right]\asymp t^{-d/2-d/\kappa}\asymp\mathrm{Var}\left[\sum_{k=1}^\infty e^{-tp_k}\right]\qquad\text{as }t\to0.\]
Definition \ref{def:Hyperuniform in the sense of semigroups}
is thus consistent with the usual interpretation of hyperuniformity, namely, large scale fluctuations smaller than that of a Poisson process with a corresponding intensity.
\end{remark}

\subsubsection{Semigroup Correlations}

Consider
\[\rho(s,t)=\mathrm{Corr}\left[\mathrm{Tr}[e^{-s H}],\mathrm{Tr}[e^{-t H}]\right]=\frac{C(s,t)}{\sqrt{C(s,s)C(t,t)}},\qquad s,t>0,\]
where $\mathrm{Corr}$ denotes the correlation.
As an immediate corollary of Theorem \ref{Thm: Main Theorem},
we obtain the following correlation estimate:

\begin{corollary}\label{cor: correlation scaling}
If \eqref{Equation: V Lower Bound} and \eqref{Equation: V Upper Bound} both hold, then
\[\rho(s,t)
\asymp\left(\frac{\min\{s,t\}}{\max\{s,t\}}\right)^{d/2\kappa}
\qquad\text{as }(s,t)\to0.\]
\end{corollary}

Recalling the relation \eqref{equation: trace count relation}, we can interpret this as a statement about the correlations of a smooth
analog of the counting function $N$ when evaluated for eigenvalues below the thresholds $1/s$ and $1/t$. The precise rate of this decorrelation reflects what we expect from physical heuristics. Indeed, we expect that the eigenvalues $\lambda_k$ and eigenfunctions $\phi_k$ of the operator $H$ are given by the solutions to the variational problem

\[\lambda_k = \inf_{\substack{\phi_k\in \mathcal Q(H), \norm{\phi_k}=1\\\phi_k\perp\phi_1,...,\phi_{k-1}}} \frac12\int_{\R^d}|\nabla\phi_k(x)|^2dx+\int_{\R^d}\phi_k(x)^2\big(\xi(x)+V(x)\big)dx,\]
where $\mathcal Q(H)$ denotes $H$'s form domain.

\begin{remark}
This equality is literal when $d=1$
since $H$ is defined through its quadratic form.
However, when $d=2$ the white
noise requires a renormalization. In that case,
the equality should be viewed as an informal heuristic.
\end{remark}

In order to minimize the first integral, the eigenfunctions are incentivized to spread out. On the other hand, in order to minimize the second integral, the eigenfunctions want to concentrate on the points where the potential and the noise are the smallest. In particular, eigenfunctions that concentrate in the same areas pick up the same noise from $\xi$ and so this competition between the two terms is what drives the eigenvalue correlations or lack thereof. Consequently, we expect a larger dimension to give the eigenfunctions $\phi_k$ more space to separate and hence decorrelate faster, whereas a stronger confining potential forces them to remain in a smaller area and ``reuse" the same noise $\xi(x)$, increasing their correlation. Corollary \ref{cor: correlation scaling} demonstrates the precise quantitative effects of these competing phenomena when measuring eigenvalue correlations through the semigroup's trace.

\begin{remark}
In \cite{FluctuationsLatticeHamiltonian} and \cite{DongJianYuan}, a more literal
manifestation of this mechanism was observed for lattice
Hamiltonians with i.i.d. noise and $-\tfrac12\Delta+\xi$ on a one-dimensional interval, respectively. More specifically, in \cite[(1.11)]{FluctuationsLatticeHamiltonian}, we see that the
above heuristic of correlations coming from ``reusing"
the same noise manifests quite literally as an eigenfunction
inner product; in \cite[Theorem 1.1]{DongJianYuan},
the random fluctuations of high eigenvalues manifest as
a stochastic Rayleigh-Schr\"odinger-type correction. Thus, Corollary \ref{cor: correlation scaling} can be viewed as a first step toward exploring this mechanism
in general Anderson models with nontrivial
noise and coercive potential in the continuum.
\end{remark}

\begin{remark}
The phenomena observed in Corollary \ref{cor: correlation scaling}
and \cite{FluctuationsLatticeHamiltonian,DongJianYuan} are in sharp contrast with
what is typically observed in full-space Anderson models with no deterministic potential (i.e. $V=0$), i.e., a complete decorrelation effect; see, e.g., \cite{DecorrelationEstimatesDiscrete,SpectralStatsLocalizedRegime}.
\end{remark}

\subsection{Open Problems}

The main results of this paper and their applications in hyperuniformity and the analysis of eigenvalue correlations lead to some natural follow-up questions. In this section, we share some of the problems and follow-ups that we consider to be the most promising.

\subsubsection{Half-Line Hyperuniformity for $d=1,2,3$}

Recall the half-line hyperuniformity ratio, $R_{\mathrm h}(\lambda)$, which we defined earlier in \eqref{eqn: HL Ratio}.
By combining Corollary \ref{cor: hyperuniformity ratio} with
the informal Abelian/Tauberian relation \eqref{equation: trace count relation}, one is naturally led to the following conjecture:

\begin{conjecture}
\label{conj: Hyperuniform}
For $d=1,2,3$, one has $R_{\mathrm h}(\lambda)\to0$
as $\lambda\to\infty$ at an explicit rate
that depends on $d$.
\end{conjecture}

While the moments of $\mathrm{Tr}[e^{-t H}]$ are known to be infinite when $d=3$ (as discussed in Section \ref{sec: Main Result}), it appears as though
the moments of $N(\lambda)$ do not suffer from the same
problem; see, e.g., \cite[Lemma 5.16]{MatsudaVanZuijlen}
(more specifically, the statement that
$\mathbb E[\overline{\boldsymbol N}^N(U,\lambda)^m]<\infty$ therein)
for a similar statement in the case of $H=-\tfrac12\Delta+\xi$ on a bounded domain.

The Weyl law proved in \cite[Proposition 5.17]{MatsudaVanZuijlen} suggests that the technology
necessary to calculate asymptotics of $\mathbb E[N(\lambda)]$
is already available. Thus, the main innovation necessary to
prove Conjecture \ref{conj: Hyperuniform} appears to be
an ability to estimate $\mathrm{Var}[N(\lambda)]$.
Toward this end, we propose the following:

\begin{problem}
\label{prob: 1}
Construct the eigenvalue counting function $N(\lambda)$
of the operator $H$ in \eqref{Equation: H}
for general choices of $V$ and $d=1,2,3$. Then, prove (or disprove) Conjecture \ref{conj: Hyperuniform}.

\end{problem}

\subsubsection{Eigenvalue Counting Correlations for $d=1,2,3$}

Similarly, we consider the eigenvalue correlation function
\[\chi(\lambda,\mu)=\mathrm{Corr}\big[N(\lambda),N(\mu)\big]=\frac{\mathrm{Cov}[N(\lambda),N(\mu)]}{\sqrt{\mathrm{Var}[N(\lambda)]\mathrm{Var}[N(\mu)]}}.\]
Using Corollary \ref{cor: correlation scaling} and the informal relation \eqref{equation: trace count relation} we arrive at the following:

\begin{conjecture}
\label{conj: correlation}
Let $d=1,2,3$.
As $(\lambda,\mu)\to\infty$,
the correlation $\chi(\lambda,\mu)$
is asymptotic to an explicit function of $\min\{\lambda,\mu\}$, $\max\{\lambda,\mu\}$, $d$, and $\kappa$.
\end{conjecture}

Again unlike the exponential trace, this conjecture
makes sense for $d=3$, which leads us to ask the following:
\begin{problem}
Following-up on Problem \ref{prob: 1},
once $N(\lambda)$ is constructed, prove (or disprove)
Conjecture \ref{conj: correlation}.
\end{problem}

\section{Definition of $\mathrm{Tr}[e^{-t H}]$} \label{Section: Trace Set Up}

We now discuss how to define the exponential trace
$\mathrm{Tr}[e^{-t H}]$. This construction is rather
different depending on whether $d=1$ or $d=2$; hence
we split the definition into two cases.

\subsection{One-Dimensional Case}

In \cite{GLFK}, the one-dimensional version
of the operator $H$ was constructed under the
assumptions of this paper using its quadratic
form; i.e.,
\begin{align}
\label{eqn: 1d form}
\mathcal Q(\phi,\psi)=\tfrac12\langle\phi',\psi'\rangle+\langle \phi\psi,V\rangle+\langle \phi\psi,\xi\rangle
\end{align}
for every $\phi,\psi:\mathbb R\to\mathbb R$ in an appropriate
weighted Sobolev space, where $\langle\phi\psi,\xi\rangle$
is interpreted as a pathwise stochastic integral
(see \cite[Definitions 2.5 and 2.8, and Propositions 2.9, all in Case 1]{GLFK}).
Then, it was shown in
\cite[Proposition 2.10]{GLFK} that, almost surely, there exists a unique self-adjoint
operator on $L^2(\mathbb R)$ whose quadratic form
is \eqref{eqn: 1d form}, and that this operator has compact
resolvent. We define $H$ as this operator. In particular, $H$'s spectrum is a sequence of
eigenvalues $\lambda_1\leq\lambda_2\leq\cdots$ without accumulation point, and the corresponding
orthonormal eigenfunctions $\phi_1,\phi_2,\ldots$ span $L^2(\mathbb R)$. One can then define $H$'s
semigroup and its trace as follows:

\begin{definition}
Let $d=1$.
For every $t>0$, we define
\[e^{-tH}\psi(x)=\sum_{k=1}^\infty e^{-t\lambda_k}\langle\psi,\phi_k\rangle\phi_k(x),
\qquad\psi\in L^2(\mathbb R),\]
and $\mathrm{Tr}[e^{-tH}]=\sum_{k=1}^\infty e^{-t\lambda_k}$.
\end{definition}

It was also shown in \cite[Theorem 2.24]{GLFK}
that $e^{-tH}$ is almost surely trace class for all $t>0$, so that
$\mathrm{Tr}[e^{-tH}]$ is well defined.

\subsection{Two-Dimensional Case}

When $d=2$, the operator $H$ must be renormalized.
The first step in the construction is to regularize
the noise:

\begin{definition}
Throughout the paper, we denote the $d$-dimensional
Gaussian kernel as
\begin{align}
\label{eqn: Gaussian kernel}
p_t(x)=\frac{e^{-|x|^2/2t}}{(2\pi t)^{d/2}},\qquad t>0,~ x\in\mathbb R^d.
\end{align}
Then, for every $\epsilon>0$, we define
$\xi_\epsilon=\xi*p_{\epsilon/2}$,
where $*$ denotes the convolution. Since the Gaussian
kernel is smooth, so is $\xi_\epsilon$. Moreover,
given that $p_s*p_t=p_{s+t}$ and $p_t(x)=p_t(-x)$ is even,
for every $\epsilon_1,\epsilon_2>0$, we have the
covariance identity
\begin{align}
\label{eqn: mixed epsilon covariance}
\mathbb E[\xi_{\epsilon_1}(x)\xi_{\epsilon_2}(y)]=p_{(\epsilon_1+\epsilon_2)/2}(y-x),
\qquad x,y\in\mathbb R^d.
\end{align}
In particular, $\xi_\epsilon$ is a centered Gaussian process with covariance kernel $p_\epsilon$.
\end{definition}

With this in hand, we may now define the
smoothed approximations
\[H_\epsilon=-\tfrac12\Delta+V+\xi_\epsilon,\qquad\epsilon>0\]
and their semigroups
using classical theory
thanks to the following (see Section \ref{section: smooth trace proofs} for proof):

\begin{lemma}
\label{Lemma: V+xi epsilon}
Let $d=2$ and $\epsilon>0$ be fixed.
\begin{enumerate}
\item Let $\kappa>0$ be as in \eqref{Equation: V Lower Bound}. There exist finite random variables $\mathfrak a_\epsilon,\mathfrak b_\epsilon>0$ such that, almost surely, one has
\[V(x)+\xi_\epsilon(x)\geq\mathfrak a_\epsilon|x|^\kappa-\mathfrak b_\epsilon,\qquad x\in\mathbb R^2.\]
\item Almost surely, $(V+\xi_\epsilon)_+\in K_2^{\mathrm{loc}}$ (recall that $K_2^{\mathrm{loc}}$ denotes the local Kato class in $\mathbb R^2$).
\end{enumerate}
\end{lemma}

Indeed, it is well-known that under these conditions,
the operator $H_\epsilon$ can be constructed through
its quadratic form, and that its semigroup is
trace class (see, e.g., \cite[Theorem A.2.7]{Simon}
for the construction of $H_\epsilon$ and $e^{-tH_\epsilon}$,
and \cite[(9.4)]{SimonBook} for the trace class condition).

At this point, one could attempt to construct
$H$ by taking the $\epsilon\to0$ limit of an appropriate renormalization of $H_\epsilon$.
More specifically,
one expects that the paracontrolled calculus/regularity
structures technology developed in \cite{AllezChouk,Labbe} to construct $-\tfrac12\Delta+\xi$ on a box could be suitably adapted
to the present setting. However,
given that we ultimately only care about $\mathrm{Tr}[e^{-tH}]$, we elect to circumvent these difficulties
entirely and construct the latter as a renormalized $L^2$ limit of the random variables $\mathrm{Tr}[e^{-t H_\epsilon}]$. Indeed, this makes it possible to construct
the trace using only probabilistic arguments, i.e., the Feynman-Kac formula. Our statement to that effect is as follows (see Section \ref{section: proof of existence of 2d trace} for proof):

\begin{prop}
\label{Proposition: Existence of 2d Trace}
Let $d=2$, and define
$c_\epsilon=\frac1{2\pi}\log(1/\epsilon)$.
There exists $\vartheta>0$ such that for
every $t\in(0,\vartheta)$, there exists a random variable $T(t)$ such that
\[\lim_{\epsilon\to0}\mathbb E\Big[\big|\mathrm{Tr}[e^{-t(H_\epsilon+c_\epsilon)}]-T(t)\big|^2\Big]=0.\]
\end{prop}

We now end this section with the following definition:

\begin{definition}
Let $\vartheta>0$ and $T(t)$ (for $t<\vartheta$)
be as in Proposition \ref{Proposition: Existence of 2d Trace}.
If $d=2$, then for every $t\in(0,\vartheta)$, we define $\mathrm{Tr}[e^{-tH}]=T(t)$.
\end{definition}

\section{Feynman-Kac Formulas and Local Times}

We now introduce and prove the Feynman-Kac formulas for the expectation and covariance of the exponential trace as well as the necessary background on the mutual and self-intersection local times of Brownian Bridges. These Feynman-Kac representations will be the key ingredient that allows us to get the precise asymptotics on the covariance of the semigroups. As discussed in Section \ref{Section: Trace Set Up}, there are some significant differences between the $d=1$ and $d=2$ cases owing to the singularity of the white noise in the latter case. As such this section is organized as follows: in Section \ref{section: LT and FK} we establish some necessary local time preliminaries and state the Feynman-Kac formulas for the expectations and covariances for both $d\in\{1,2\}$. In Section \ref{section: FK d=1 proof}, we prove the expectation and covariance formulas for $d=1$ and in Section \ref{section: smooth trace proofs} we prove the existence of a Feynman-Kac formula for the smoothed exponential trace when $d=2$. In Section \ref{section: LT prelims}, we collect largely standard results on local times that we use for proving the expectation and covariance formulas when $d=2$ and we conclude in Sections \ref{section: proof of existence of 2d trace} and \ref{section: d=2 expectation and covariance proofs} by proving the convergence of the smoothed exponential trace and the final expectation and covariance formulas. 

\begin{definition}
\label{def: BB Couplings}
We take $B_1^{0,0}$ to mean a standard Brownian Bridge in $\R^d$ for $d\in\{1,2\}$ that satisfies $B^{0,0}_1(0)=B_1^{0,0}(1)=0$. For $x\in\R^d$ and $t>0$, we define $B_t^{x,x}(u)$ by the coupling \[B^{x,x}_t(u) = x+\sqrt{t}B_1^{0,0}(u/t)\] In words, this is a Brownian Bridge from $x$ to $x$ in time $t$ with one specific choice of coupling over all $x,t$. Similarly, we use $\overline{B}^{y,y}_s$ to denote an independent Brownian Bridge,
with the same coupling over $y\in\mathbb R^d$ and $s>0$.
\end{definition}

\subsection{Definitions and Formulas}\label{section: LT and FK}
For $d=1$, $0\leq s\leq t$ and $a\in\R$, let $L^a_{[s,t]}(B_t^{x,x})$ denote the continuous version of the local time of $B_t^{x,x}$ at the point $a$ collected in the time interval $[s,t]$. For any measurable function $f:\R\to\R$ satisfying appropriate integrability conditions, the local time gives the representation,
\begin{align*}
\int_s^tf(B_t^{x,x}(u))du = \int_\R L_{[s,t]}^a(B_t^{x,x})f(a)da = \inn{L_{[s,t]}(B_t^{x,x}),f}.
\end{align*}

The self-intersection local time (SILT) of a Brownian Bridge with itself and the mutual intersection local time (MILT) of two independent Brownian Bridges $B_t^{x,x}$ and $\overline{B}_s^{y,y}$ will be essential quantities for the remainder of this paper. We can intuitively think of the self-intersection local time as a function that counts how often the bridge $B_t^{x,x}$ intersects itself, whereas the mutual intersection local time counts how many times the bridges $B_t^{x,x}$ and $\overline{B}_s^{y,y}$ intersect each other. When $d=1$, we can rigorously define these two quantities in terms of the local time of Brownian Bridges as follows:

\begin{definition}\label{def: MILT and SILT d=1}
For $d=1$ and $t>0$ the self-intersection local time of a Brownian Bridge $B_t^{x,x}$ is defined as 
\begin{align}
    \beta_t(B_t^{x,x}) = \int_{-\infty}^\infty (L_t^a(B_t^{x,x}))^2da = \inn{L_t(B_t^{x,x}),L_t(B_t^{x,x})} = \norm{L_t(B_t^{x,x})}_2^2
\end{align}
and for $s,t>0$ the mutual intersection local time of the (independent) Brownian Bridges $B_t^{x,x}$ and $\overline{B}_s^{y,y}$ is defined as
\begin{align}
    \alpha_{t,s}(B_t^{x,x},\overline{B}_s^{y,y}) = \int_{-\infty}^\infty  L_t^a(B^{x,x}_t)L_s^a(\overline{B}^{y,y}_s)da = \inn{L_t^a(B_t^{x,x}),L_s^a(\overline{B}_s^{y,y})}
\end{align}
respectively.
\end{definition}

When $d=2$, we are no longer able to define the MILT and the SILT using the local time and so instead we introduce the following approximations in order to rigorously define the two quantities.  

\begin{definition}\label{def: approx MILT and SILT 2D}
For $d\in\{1,2\}$, $\epsilon>0$, and $A\subset[0,t]^2$ a bounded Borel subset, the approximate self-intersection local time of a Brownian Bridge, $B_t^{x,x}$, is defined as 
\begin{align}
\label{eqn: approx silt}
    \beta_A^\epsilon(B^{x,x}_t) =  \int_Ap_\epsilon(B_t^{x,x}(u)-B_t^{x,x}(v))dudv
\end{align}
(note that $\beta_A^\epsilon(B^{x,x}_t)= \beta_A^\epsilon(B^{0,0}_t)$).
For $A\subset[0,t]\times[0,s]$, the approximate mutual intersection local time of the Brownian Bridges $B_t^{x,x}$ and $\overline{B}_s^{y,y}$ is defined as
\begin{align}\label{eqn: approx milt}
    \alpha^\epsilon_A(B_t^{x,x},\overline{B}_s^{y,y}) = \int_Ap_\epsilon(B_t^{x,x}(u)-\overline{B}_s^{y,y}(v))dudv
\end{align}
\end{definition}
Using these we can now define the MILT and SILT in two dimensions by the following propositions that tell us that these approximations have well defined limits. Beginning with the mutual intersection time we have:

\begin{prop}\label{Prop: MILT convergence}
Let $d=2$, $s,t>0$, and $x,y\in\R^d$ be fixed. Then for any Borel measurable $A_1\times A_2\subset[0,t]\times[0,s]$ there exists a random variable, $\alpha_{A_1\times A_2}$, such that 
\begin{align*}
    \alpha_{A_1\times A_2}(B_t^{x,x},\overline{B}_s^{y,y}) = \lim_{\epsilon\to0} \alpha^\epsilon_{A_1\times A_2}(B_t^{x,x},\overline{B}_s^{y,y}) \qquad \text{in probability}
\end{align*}
\end{prop}
\begin{proof}
    \cite[Theorem 5.11]{GLP}.
\end{proof}

For the self-intersection time we use the following notation: Given any $A\subset\R^2$ let \[A_{\leq}=\{(x,y)\in A:x\leq y\}.\]

Using this we now get the convergence of the approximate SILT to a well defined limit.
\begin{prop}\label{prop: approx SILT}
Let $d=2$, $x\in\R^d$ and $t>0$ be fixed.

For  $0\leq a<b\leq t$, there exist a random variable $\gamma_{[a,b]^2_{\leq}}(B_t^{x,x})$ such that
\[\gamma_{[a,b]^2_{\leq}}(B_t^{x,x})=\lim_{\epsilon\to0}(\beta^\epsilon_{[a,b]^2_\leq}(B_t^{x,x})-\E[\beta^\epsilon_{[a,b]^2_{\leq}}(B_t^{x,x})]) \qquad\text{in probability}\]
   
\end{prop}
\begin{proof}
\cite[Theorem 3.7]{Matsuda}.
\end{proof}

\begin{notation}
Throughout the paper we use the following shorthands:
\[\alpha_t^{\epsilon}(\cdot)=\alpha^\epsilon_{[0,t]^2}(\cdot), \qquad \alpha_{t,s}^{\epsilon}(\cdot)=\alpha^\epsilon_{[0,t]\times[0,s]}(\cdot),\qquad\beta_t^{\epsilon}(B)=\beta^\epsilon_{[0,t]^2_{\leq}}(B)\]
as well as the obvious extensions of these notations for the $\epsilon\to0$ limits, namely, $\alpha_t(\cdot)$, $\alpha_{t,s}(\cdot)$,
and $\gamma_t(\cdot)$.
\end{notation}

In Proposition \ref{prop: approx SILT}, it is necessary to normalize the SILT in $\R^2$ by subtracting the mean in order to get a non-trivial limit. Indeed, the following Lemma illustrates this:
\begin{lemma}\label{Lem: Approx Silt Expectation}
    For $t>0$ fixed, $0\leq a<b\leq t$, and $\epsilon\to0$ we have
    \begin{align*}
        \E[\beta^\epsilon_{[a,b]_{\leq}^2}(B_t^{0,0})]=\frac{b-a}{2\pi}(\log(1/\epsilon)+\log(t))+\frac{1}{2\pi}\int_0^{b-a}\log(r)-\log(t-r)dr+o(1)
    \end{align*}
\end{lemma}
\begin{proof}
    See \cite[Lemma 3.4]{Matsuda}.
\end{proof}
To compress and unify the expectation and covariance formulas across dimensions we introduce the following shorthands.
\begin{definition}
\label{def: abc shorthands}
Let $d\in\{1,2\}$, $x,y\in\mathbb R^d$,
and $s,t>0$. Denote
\begin{align*}
\mathcal{P}^d(t,s) &= \begin{cases}\frac{1}{(2\pi t)^{1/2}(2\pi s)^{1/2}}&d=1\\ \frac{e^{t\log t/2\pi}e^{s\log s/2\pi}}{(2\pi t)(2\pi s)}&d=2
\end{cases}\\ 
\mathcal{A}_{t,s}(x,y)&= -\int_0^tV(B_t^{x,x}(u))du-\int_0^sV(\overline{B}_s^{y,y}(u))du\\
\mathcal{B}^d_{t,s}(x,y) &= 
\begin{cases}\frac{1}{2}(\beta_{t}(B_t^{x,x})+\beta_{s}(\overline{B}_s^{y,y})) & d=1\\
\gamma_{t}(B_t^{x,x})+\gamma_{s}(\overline{B}_s^{y,y}) & d=2
\end{cases}\\
\mathcal{C}_{t,s}(x,y)&=\alpha_{t,s}(B_t^{x,x},\overline{B}_s^{y,y})\\
\mathcal{D}^d_t(x)  &= 
\begin{cases}\frac{1}{2}\beta_{t}(B_t^{x,x}) & d=1\\
\gamma_{t}(B_t^{x,x}) & d=2
\end{cases}
\end{align*}
\end{definition}
We are now ready to state formulas for the expectation and covariance of the exponential trace
(see Section \ref{section: FK d=1 proof} for proof when $d=1$ and Section \ref{section: d=2 expectation and covariance proofs} for proof when $d=2$):
\begin{prop}\label{prop: 2D FK expectation and covariance}
For $d=2$, there exists $\vartheta>0$ such that for every $s,t\in(0,\vartheta)$
    \begin{align}\label{eqn: FK Expectation}
        \E[\mathrm{Tr}[e^{-tH}]] = \mathcal{P}^d(t,t)^{1/2}\int_{\R^d}\E[e^{-\int_0^tV(B_t^{x,x}(u))du+\mathcal{D}^d_t(x)}]dx.
    \end{align}
    \begin{align}\label{eqn: FK Covariance}
        \mathrm{Cov}\big[\mathrm{Tr}[e^{-tH}],\mathrm{Tr}[e^{-sH}]\big] = \mathcal{P}^d(t,s)\int_{(\R^d)^2}\E\left[(e^{\mathcal{A}_{t,s}(x,y)+B^d_{t,s}(x,y)})(e^{\mathcal{C}_{t,s}(x,y)}-1)\right]dxdy
    \end{align}
For $d=1$, the same formulas hold for every $s,t>0$.
\end{prop}

\subsection{Proof of Proposition \ref{prop: 2D FK expectation and covariance}; $d=1$}\label{section: FK d=1 proof}

Let $d=1$.
The expectation formula follows immediately from taking the expectation of the Feynman-Kac formula for the exponential trace proven in \cite[Theorem 2.24]{GLFK}. In particular, it is shown that \[\mathrm{Tr}[e^{-tH}]=\frac{1}{\sqrt{2\pi t}}\int_\R\E_B\left[e^{-\int_0^tV(B_t^{x,x}(u))du-\xi(L_t(B_t^{x,x}))}\right]dx,\]
where $\xi(f)$ denotes stochastic integration of $f$ with respect to a Brownian motion $W$ such that $\xi=W'$, $\mathbb E_B$ denotes the expectation with respect to $B^{x,x}_t$ only (hence, conditional on $\xi$), and we assume that $\xi$ and $B^{x,x}_t$ are independent.
The expectation formula \eqref{eqn: FK Expectation} then follows immediately from an application of Fubini's Theorem and the fact that $\E[e^{-\xi(f)}]=e^{\frac{1}{2}\norm{f}^2}$ (see \cite[Section 4.5]{LGL} where this is done explicitly). 

The covariance formula then follows from a simple computation using \[\mathrm{Cov}\big[\mathrm{Tr}[e^{-tH}\big],\mathrm{Tr}[e^{-sH}]]=\E\big[\mathrm{Tr}[e^{-tH}]\cdot\mathrm{Tr}[e^{-sH]}]\big]-\E\big[\mathrm{Tr}[e^{-tH}]\big]\E\big[\mathrm{Tr}[e^{-sH}]\big].\] 
Observe that another application of Fubini's Theorem gives us
\[\E\big[\mathrm{Tr}[e^{-tH}]\cdot\mathrm{Tr}[e^{-sH}]\big] = \frac{1}{2\pi\sqrt{ts}}\int_{\R^2}\E[e^{\mathcal{A}_{t,s}(x,y)+\mathcal{B}^1_{t,s}(x,y)+\mathcal{C}_{t,s}(x,y)}]dxdy\]
Subtracting $\E\big[\mathrm{Tr}[e^{-tH}]\big]\E\big[\mathrm{Tr}[e^{-sH}]\big]$ then gives us the desired covariance formula.

\subsection{Feynman-Kac for Smooth Trace; $d=2$ (Proof of Lemma \ref{Lemma: V+xi epsilon})}\label{section: smooth trace proofs}

It is well-known that if Lemma \ref{Lemma: V+xi epsilon} holds, then for every $\epsilon,t>0$,
\begin{align}
\label{eqn: d=2 epsilon raw Feynman-Kac}
\mathrm{Tr}[e^{-t H_\epsilon}]=\frac1{2\pi t}\int_{\mathbb R^2}\mathbb E_B\left[\exp\left(-\int_0^tV\big(B^{x,x}_t(s)\big)+\xi_\epsilon\big(B^{x,x}_t(s)\big)ds\right)\right]dx<\infty
\end{align}
almost surely; see \cite[(9.4) and (9.5)]{SimonBook}.
Similarly to $d=1$, we assume in \eqref{eqn: d=2 epsilon raw Feynman-Kac} that $B^{x,x}_t$ and $\xi_\epsilon$ are independent, and that $\mathbb E_B$ denotes the expectation with respect to $B^{x,x}_t$ only.
Thus, it suffices to prove Lemma \ref{Lemma: V+xi epsilon}. 
Toward this end, we record two key properties:
For every $\epsilon>0$, it holds that
\begin{enumerate}[i.]
\item $\xi_\epsilon$ is continuous almost surely (this is immediate by definition of $\xi_\epsilon$ as the convolution of $\xi$ against a smooth function), and
\item there exists a random variable $C_\epsilon>0$ such that, almost surely,
\begin{align}
\label{eqn: xi eps log bound}
\sup_{|x|\leq k}|\xi_\epsilon(x)|\leq C_\epsilon\sqrt{\log k},\qquad k=2,3,\ldots.
\end{align}
This is a straightforward consequence of standard Gaussian suprema estimates and the Borel-Cantelli lemma; e.g., \cite[(3.1)]{Ortega}.
\end{enumerate}
On the one hand, if we combine
\eqref{Equation: V Lower Bound} and \eqref{eqn: xi eps log bound}, then we obtain that
$V(x)+\xi_\epsilon(x)\geq\mathfrak a_\epsilon|x|^\kappa-\mathfrak b_\epsilon$
almost-surely for some random $\mathfrak a_\epsilon,\mathfrak b_\epsilon>0$.
On the other hand, if we combine the facts that
$(V+\xi_\epsilon)_+\leq V_++(\xi_\epsilon)_+$, that $V_+\in K_2^{\mathrm{loc}}$ (by assumption), that $K_2^{\mathrm{loc}}$ is a vector space (e.g., \cite[Page 454]{Simon}), and that every continuous function is locally Kato (since it is locally bounded; see \cite[(A.13)]{Simon}),
then we obtain that $(V+\xi_\epsilon)_+\in K_2^{\mathrm{loc}}$ whenever $\xi_\epsilon$ is continuous,
which occurs with probability one.
This simultaneously completes the proof
of \eqref{eqn: d=2 epsilon raw Feynman-Kac} and Lemma \ref{Lemma: V+xi epsilon}.

\subsection{Local Time Preliminaries}\label{section: LT prelims}
In this section, we compile some standard results on the intersection local times of Brownian Bridges, which will be needed to prove Propositions \ref{Proposition: Existence of 2d Trace} and \ref{prop: 2D FK expectation and covariance}.

We begin with the following uniform integrability estimates which will be important for proving our Feynman-Kac formulas as well as the convergence of the smoothed trace.
\begin{prop}\label{Prop: MILT UI}
If $d=1$, then for every $c,s,t>0$,
\[\sup_{x,y\in\mathbb R}\E\big[\exp(c\alpha_{t,s}(B^{x,x}_t,\overline B^{y,y}_s))\big]<\infty.\]
    If $d=2$, then for every $c>0$ there exists $\theta_c>0$ such that for every $s,t\in(0,\theta_c)$ we have
    \[\sup_{\epsilon>0,\,x,y\in\R^2}\E\big[\exp(c\alpha^\epsilon_{t,s}(B_t^{x,x},\overline{B}_s^{y,y}))\big]\leq\sup_{x,y\in\R^2}\E[\exp(c\alpha_{t,s}(B_t^{x,x},\overline{B}_s^{y,y}))]<\infty.\]
\end{prop} 
\begin{proof}
For $d=1$, note that
\begin{multline}
\label{eqn: 1D MILT quadratic form inequality}
\alpha_{t,s}(B_t^{x,x},\overline{B}_s^{y,y}) = \int_{-\infty}^\infty  L_t^a(B^{x,x}_t)L_s^a(\overline{B}^{y,y}_s)da\\
\leq\frac12\big(\|L_t(B^{x,x}_t)\|_2^2+\|L_s(\overline{B}^{y,y}_s)\|_2^2\big)
=\frac{\beta_{t}(B_t^{x,x})+\beta_s(\overline{B}_s^{y,y})}2.
\end{multline}
The result then follows from Proposition \ref{Prop: 1D SILT Bound} below.
    For $d=2$, the result is proved in \cite[Lemma 5.15]{GLP}. 
\end{proof}

We take this opportunity to now record two propositions that give us exponential bounds on the SILT in one and two dimensions, respectively, which we will need later for asymptotically bounding the semigroup covariance.

\begin{prop}\label{Prop: 1D SILT Bound}
    Let $d=1$. For every $c,t>0$,
    \begin{align}\label{eqn: 1D SILT Bound}
        \sup_{x\in\mathbb R}\E[\exp(c\beta_{t}(B_t^{x,x}))]<\infty.
    \end{align}
    
\end{prop}
\begin{proof}
\cite[Lemma 5.11]{GLFK}.
\end{proof}

\begin{prop}\label{Prop: 2D SILT Bound}
    Let $d=2$. For every $c>0$, there exists $\theta_c>0$ such that for every $t\in(0,\theta_c)$,
    \begin{align}\label{eqn: 2D SILT Bound}
        \sup_{x\in\R^2}\E[\exp(c\gamma_t(B_t^{x,x}))]<\infty.
    \end{align}
\end{prop}
\begin{proof}
\cite[Remark 5.8]{GLP}.
\end{proof}

Finally, we conclude with the following uniform integrability estimate which will be important for proving our Feynman-Kac formulas.
\begin{prop}\label{Prop: SILT UI}
Let $d=2$. For every $c>0$ there exists $\theta_c>0$ such that for every $t\in(0,\theta_c)$
\begin{align*}
    \sup_{x\in\R^2,\,\epsilon>0}\E[\exp(c(\beta_t^\epsilon(B_t^{x,x})-\E[\beta_t^\epsilon(B_t^{x,x})])]<\infty.
\end{align*}
\end{prop}
\begin{proof}
    See \cite[Lemmas 5.6 and 5.7]{GLP}.
\end{proof}

\subsection{Smoothed Trace Convergence (Proof of Proposition \ref{Proposition: Existence of 2d Trace})}\label{section: proof of existence of 2d trace}
We are now prepared to prove Proposition \ref{Proposition: Existence of 2d Trace}. 
Given that the $L^2$ space of square integrable random variables is complete, it suffices to prove that there exists $\vartheta>0$ such that the limit
\begin{align}
\label{eqn: L2 Limit for d=2}
\lim_{(\epsilon_1,\epsilon_2)\to0}\mathbb E\Big[\mathrm{Tr}[e^{-t(H_{\epsilon_1}+c_{\epsilon_1})}]\mathrm{Tr}[e^{-t(H_{\epsilon_2}+c_{\epsilon_2})}]\Big]\end{align}
exists and is finite for all $t\in(0,\vartheta)$.

By applying Tonelli's theorem to \eqref{eqn: d=2 epsilon raw Feynman-Kac},
for every $s,t,\epsilon_1,\epsilon_2>0$,
almost surely,
\begin{multline}
\label{eqn: d=2 epsilon raw Feynman-Kac for 2nd moment}
\mathrm{Tr}[e^{-t H_{\epsilon_1}}]
\mathrm{Tr}[e^{-s H_{\epsilon_2}}]=\frac1{4\pi^2st}\int_{(\mathbb R^2)^2}\mathbb E_B\bigg[\exp\bigg(-\int_0^tV\big(B^{x,x}_t(u)\big)+\xi_{\epsilon_1}\big(B^{x,x}_t(u)\big)du\\
-\int_0^sV\big(\overline B^{y,y}_s(u)\big)+\xi_{\epsilon_2}\big(\overline B^{y,y}_s(u)\big)du\bigg)\bigg]dxdy,
\end{multline}
where we assume that $\xi$, $B^{x,x}_t$, and $\overline{B^{y,y}_s}$
are independent.
Since $\xi_\epsilon$ is a Gaussian process with
mean zero and covariance kernel function $p_\epsilon$,
conditional on a fixed realization of $B^{x,x}_t$
and $\overline{B^{y,y}_s}$, the integral 
\[\int_0^t\xi_{\epsilon_1}\big(B^{x,x}_t(u)\big)du+\int_0^s\xi_{\epsilon_2}\big(\overline B^{y,y}_s(u)\big)du\]
is Gaussian with mean zero and variance
\[\mathbb E_\xi\left[\left(\int_0^t\xi_{\epsilon_1}\big(B^{x,x}_t(u)\big)du+\int_0^s\xi_{\epsilon_2}\big(\overline B^{y,y}_s(u)\big)du\right)^2\right],\]
where $\E_\xi$ denotes the expectation with respect to $\xi$ only,
conditional on $B^{x,x}_t$ and $\overline B^{y,y}_s$.
If we expand this variance using \eqref{eqn: mixed epsilon covariance} and Fubini's theorem (the latter of which we can apply since $\xi_\epsilon$ is almost surely continuous hence bounded on compacts),
then we get
\begin{multline*}
\int_0^t\int_0^tp_{\epsilon_1}\big(B_t^{x,x}(u)-B_t^{x,x}(v)\big)dudv
+\int_0^t\int_0^s2p_{(\epsilon_1+\epsilon_2)/2}\big(B_t^{x,x}(u)-\overline B_s^{y,y}(v)\big)dudv\\
+\int_0^s\int_0^sp_{\epsilon_2}\big(\overline B_s^{y,y}(u)-\overline B_s^{y,y}(v)\big)dudv
=2\beta_t^{\epsilon_1}(B^{x,x}_t)
+2\alpha_{t,s}^{(\epsilon_1+\epsilon_2)/2}(B^{x,x}_t,\overline B^{y,y}_s)
+2\beta_s^{\epsilon_2}(\overline B^{y,y}_s),
\end{multline*}
where equality follows from \eqref{eqn: approx silt}
and \eqref{eqn: approx milt}. Therefore,
a straightforward application of Tonelli's theorem
and a Gaussian moment generating
function calculation yields
\begin{multline}
\label{eqn: d=2 epsilon trace inner product formula}
\mathbb E\Big[\mathrm{Tr}[e^{-t(H_{\epsilon_1}+c_{\epsilon_1})}]\mathrm{Tr}[e^{-s(H_{\epsilon_2}+c_{\epsilon_2})}]\Big]\\
=\frac{1}{4\pi^2st}\int_{(\mathbb R^2)^2}\mathbb E\bigg[\exp\bigg(-\int_0^tV\big(B^{x,x}_t(u)\big)du-\int_0^sV\big(\overline B^{y,y}_s(u)\big)du\\
+\big(\beta_t^{\epsilon_1}(B^{x,x}_t)-t c_{\epsilon_1}
\big)
+\big(\beta_s^{\epsilon_2}(\overline B^{y,y}_s)-sc_{\epsilon_2}
\big)
+\alpha^{(\epsilon_1+\epsilon_2)/2}_{t,s}(B^{x,x}_t,\overline B^{y,y}_s)\bigg)\bigg]dxdy.
\end{multline}

If we combine Proposition \ref{prop: approx SILT} with
Lemma \ref{Lem: Approx Silt Expectation} alongside the facts that $c_\epsilon=\frac{1}{2\pi}\log(1/\epsilon)$ and $\int_0^t\log(r)-\log(t-r)dr=0$, then we get
\begin{align}\label{eqn: normalized approx SILT limit}
    \lim_{\epsilon\to0}\beta^\epsilon_t(B_t^{x,x})-tc_\epsilon = \gamma_t(B_t^{x,x})+\frac{t\log t}{2\pi}\qquad\text{in probability}.
\end{align}
Together with Proposition \ref{Prop: MILT convergence},
this implies (recall the notations in Definition \ref{def: abc shorthands})
\begin{multline}\label{eqn: approx product expectation limit}
    \lim_{(\epsilon_1,\epsilon_2)\to0}\mathbb E\Big[\mathrm{Tr}[e^{-t(H_{\epsilon_1}+c_{\epsilon_1})}]\mathrm{Tr}[e^{-s(H_{\epsilon_2}+c_{\epsilon_2})}]\Big] \\
    = \mathcal P^2(t,s)\int_{(\R^2)^2}\E[e^{\mathcal{A}_{t,s}(x,y)+\mathcal{B}^2_{t,s}(x,y)+\mathcal{C}_{t,s}(x,y)}]dxdy,
\end{multline}
provided we can take $\lim_{(\epsilon_1,\epsilon_2)\to0}$ inside the $dxdy$ integral and the expectation in \eqref{eqn: d=2 epsilon trace inner product formula}.

By dominated convergence, in order to first take the limit inside the integral,
it suffices to find an integrable function that dominates the expectation in \eqref{eqn: d=2 epsilon trace inner product formula}.
Toward this end, for $x,\overline{x}\in\R^d$ and $\kappa>0$, we have the elementary inequality 
\begin{align}\label{eqn: min triangle inequality}
|x+\overline{x}|^\kappa\geq|x+\overline{x}|^{\min{\{\kappa,1\}}}-1\geq|x|^{\min{\{\kappa,1\}}}-|\overline{x}|^{\min{\{\kappa,1\}}}-1.
\end{align} 
Using the lower bound \eqref{Equation: V Lower Bound}, this yields (henceforth writing $\min\{\kappa,1\}=\kappa\land 1$)
\begin{multline}\label{eqn: V uniform bound}
\int_0^tV(B_t^{x,x}(u))du\geq
\int_0^ta|x+B_t^{0,0}(u)|^\kappa du-bt\\
\geq
a|x|^{\kappa\land1}t-a\max_{u\in[0,t]}|B_t^{0,0}(u)|^{\kappa\land1}t-(a+b)t.
\end{multline}
Thus, up to the harmless constant $e^{(a+b)(s+t)}$, we can dominate the expectation in \eqref{eqn: d=2 epsilon trace inner product formula}
for all $\epsilon\in(0,1)$ by
\begin{multline}
\label{eqn: DCT in L2 Proof}
e^{-at|x|^{\kappa\land1}-as|y|^{\kappa\land1}}
\sup_{\epsilon_1,\epsilon_2\in(0,1),~x,y\in\mathbb R^2}\mathbb E\bigg[\exp\bigg(a\max_{u\in[0,t]}|B_t^{0,0}(u)|^{\kappa\land1}t+a\max_{u\in[0,s]}|\overline B_s^{0,0}(u)|^{\kappa\land1}s\\
+\big(\beta_t^{\epsilon_1}(B^{x,x}_t)-t c_{\epsilon_1}
\big)
+\big(\beta_s^{\epsilon_2}(\overline B^{y,y}_s)-s c_{\epsilon_2}
\big)
+\alpha^{(\epsilon_1+\epsilon_2)/2}_{t,s}(B^{x,x}_t,\overline B^{y,y}_s)\bigg)\bigg].
\end{multline}
By combining an application of H\"older's inequality with the fact that $\displaystyle\max_{u\in[0,t]}|B_t^{0,0}(u)|$ has finite exponential moments (see e.g. \cite[Remark 3.1]{BesselBridgeExponentialMoments}), the uniform bounds
in Propositions \ref{Prop: MILT UI} and \ref{Prop: SILT UI}, and the fact that
\[\sup_{\epsilon\in(0,1),~x\in\mathbb R^2}\mathbb E\big[\beta_t^{\epsilon}(B^{x,x}_t)\big]-tc_\epsilon<\infty\]
(by Lemma \ref{Lem: Approx Silt Expectation}),
we conclude that the supremum of expectations
in \eqref{eqn: DCT in L2 Proof} is finite, provided $s,t<\vartheta$ for an appropriate $\vartheta$.
Thus, \eqref{eqn: DCT in L2 Proof} is integrable for small $s,t$,
as desired.

Finally, in order to take the limit inside
the expectation, it suffices to show that
the random variables inside the expectation
in \eqref{eqn: d=2 epsilon trace inner product formula} are
uniformly integrable for every choice of $x,y\in\mathbb R^2$. For this purpose, it
suffices to find $p>1$ such that, for small enough $s,t>0$,
\begin{multline}
\sup_{\epsilon_1,\epsilon_2\in(0,1)}
\mathbb E\bigg[\exp\bigg(-p\int_0^tV\big(B^{x,x}_t(u)\big)du-p\int_0^sV\big(\overline B^{y,y}_s(u)\big)du\\
+p\big(\beta_t^{\epsilon_1}(B^{x,x}_t)-t c_{\epsilon_1}
\big)
+p\big(\beta_s^{\epsilon_2}(\overline B^{y,y}_s)-s c_{\epsilon_2}
\big)
+p\alpha^{(\epsilon_1+\epsilon_2)/2}_{t,s}(B^{x,x}_t,\overline B^{y,y}_s)\bigg)\bigg]<\infty.\end{multline}
Using \eqref{eqn: V uniform bound} and H\"older's inequality, this is proved in exactly the same way
as the finiteness of the supremum in \eqref{eqn: DCT in L2 Proof}. The proof of Proposition \ref{Proposition: Existence of 2d Trace} is thus complete.

\subsection{Proof of Proposition \ref{prop: 2D FK expectation and covariance}; $d=2$}\label{section: d=2 expectation and covariance proofs}
Let $d=2$. This proof follows nearly identically to the one in Section \ref{section: proof of existence of 2d trace}.

Recall that Proposition \ref{Proposition: Existence of 2d Trace} gives us the $L^2$ convergence of the smoothed trace for every $s,t\in(0,\vartheta)$ for some $\vartheta>0$. An immediate consequence of this is that for every $s,t\in(0,\vartheta)$,
\begin{align}
    \E[\mathrm{Tr}[e^{-tH}]]&=\lim_{\epsilon\to0}\E[\mathrm{Tr}[e^{-t(H_\epsilon+c_\epsilon)}]] \\ &\text{and}\notag \\ \mathrm{Cov}[\mathrm{Tr}[e^{-tH}],\mathrm{Tr}[e^{-sH}]] &= \lim_{\epsilon\to0}  \mathrm{Cov}[\mathrm{Tr}[e^{-t(H_\epsilon+c_\epsilon)}],\mathrm{Tr}[e^{-s(H_\epsilon+c_\epsilon)}]].
\end{align}
Therefore, the desired result follows if we are able to show 
\begin{align}
\label{eqn: trace formula limit 1}
     \lim_{\epsilon\to0}\E[\mathrm{Tr}[e^{-t(H_\epsilon+c_\epsilon)}]]&= \mathcal{P}^2(t,t)^{1/2}\int_{\R^2}\E[e^{-\int_0^tV(B_t^{x,x}(u))du+\gamma_t(B_t^{x,x})}]dx
     \end{align}
\begin{multline}     \label{eqn: trace formula limit 2}
     \lim_{\epsilon\to0}  \mathrm{Cov}[\mathrm{Tr}[e^{-t(H_\epsilon+c_\epsilon)}],\mathrm{Tr}[e^{-s(H_\epsilon+c_\epsilon)}]]\\= \mathcal{P}^2(t,s)\int_{(\R^2)^2}\E\left[(e^{\mathcal{A}_{t,s}(x,y)+\mathcal B^2_{t,s}(x,y)})(e^{\mathcal{C}_{t,s}(x,y)}-1)\right]dxdy.
     \end{multline}
In order to prove these limits, we recall the Feynman-Kac formula for the smoothed exponential trace \eqref{eqn: d=2 epsilon raw Feynman-Kac}.
By applying the same Gaussian moment generating function that led to \eqref{eqn: d=2 epsilon trace inner product formula} (only easier since there is now one trace), we get
\begin{align}\label{eqn: approx expectation formula}
     \E\big[\mathrm{Tr}[e^{-t(H_\epsilon+c_\epsilon)}]\big] = \frac{1}{2\pi t}\int_{\R^2}\E\big[e^{-\int_0^tV(B^{x,x}_t(s))ds+\beta_t^\epsilon(B_t^{x,x})-tc_\epsilon}\big]dx.
\end{align}
With this in hand, a simpler version of the argument that
led to \eqref{eqn: approx product expectation limit}
allows to take $\epsilon\to0$ in \eqref{eqn: approx expectation formula},
and yields the limit \eqref{eqn: trace formula limit 1}. In fact, by an application of Tonelli's theorem,
we also have that
\begin{align}
\label{eqn: trace formula limit 1 squared}
 \lim_{\epsilon\to0}\E[\mathrm{Tr}[e^{-t(H_\epsilon+c_\epsilon)}]]\E[\mathrm{Tr}[e^{-s(H_\epsilon+c_\epsilon)}]]
=\mathcal{P}^2(t,s)\int_{(\R^2)^2}\E\left[e^{\mathcal{A}_{t,s}(x,y)+\mathcal B^2_{t,s}(x,y)}\right]dxdy.
\end{align}
Our desired covariance formula \eqref{eqn: trace formula limit 2} then follows immediately from
\eqref{eqn: trace formula limit 1 squared}
and
\eqref{eqn: approx product expectation limit},
and the fact that $\mathrm{Cov}[X,Y]=\mathbb E[XY]-\mathbb E[X]\mathbb E[Y]$.

\section{Proof of Main Results}

\subsection{Outline}

Recall the covariance formula \eqref{eqn: FK Covariance}. We observe that by Taylor expanding we can approximate $e^{\mathcal C_{t,s}(x,y)}-1\approx\mathcal C_{t,s}(x,y)$. Furthermore,  we intuitively expect $e^{\mathcal B^d_{t,s}(x,y)}$ to tend to $1$ independently of $x,y$ as $(s,t)\to0$ since there can of course be no self-intersection if $s=t=0$. Finally, since the Brownian Bridge is bounded on any finite time interval we expect $\int_{(\R^d)^2}e^{\mathcal A_{t,s}(x,y)}dxdy$ to be bounded above and below by $\int_{(\R^d)^2}e^{-t|x|^\kappa-s|y|^\kappa}dxdy$ as $(s,t)\to0$ (up to some harmless constants). 

Our arguments will formalize and quantify these heuristics. Putting this aside for now, however, we see that the main estimate we will need is contained in the following technical lemma,
which controls the terms we expect to be dominant:

\begin{lemma}\label{lem: t+s lower bound}
Let $d\in\{1,2\}$, $C,\kappa>0$, and $0<\eta\leq1$. Then as $(s,t)\to0$, we have
\begin{enumerate}[a.]
\item $\displaystyle\int_{\R^d}\int_{\R^d}e^{-C(t^\eta|x|^\kappa+s^\eta|y|^\kappa)}\E[\alpha_{t,s}(B_t^{x,x},\overline{B}_s^{y,y})]dxdy\asymp ts(t+s)^{-d\eta/\kappa}.$
\item $\displaystyle\E\left[\left(\int_{\R^d}\int_{\R^d}e^{-C(t^\eta|x|^\kappa+s^\eta|y|^\kappa)}\alpha_{t,s}(B_t^{x,x},\overline{B}_s^{y,y})dxdy\right)^2\right]\asymp t^2s^2(t+s)^{-2d\eta/\kappa}.$
\end{enumerate}
\end{lemma}

\begin{proof}
See Section \ref{Proof of lemma} for proof.
\end{proof}
\subsection{Upper Bound}\label{upper bound proof}
We begin by proving \eqref{eq: Main Upper Bound}. By using the formula \eqref{eqn: FK Covariance}, along with the inequality $e^x-1\leq xe^x$ for $x\geq0$ and the fact that $\mathcal{P}^d(t,s)\lesssim(st)^{-d/2}$
(since $e^{t\log t/2\pi+s\log s/2\pi}=1+o(1)$ as $(s,t)\to0$),
we obtain that
\begin{align}
\notag C(s,t) &= 
\int_{(\R^d)^2}\mathcal{P}^d(t,s)\E[\mathrm{e}^{\mathcal{A}_{t,s}(x,y)+\mathcal{B}^d_{t,s}(x,y)}(e^{\mathcal{C}_{t,s}(x,y)}-1)]dxdy\\&\lesssim 
\frac{1}{(st)^{d/2}}\int_{(\R^d)^2}\E[\mathrm{e}^{\mathcal{A}_{t,s}(x,y)+\mathcal{B}^d_{t,s}(x,y)}\mathcal{C}_{t,s}(x,y)e^{\mathcal C_{t,s}(x,y)}]dxdy.
\label{eqn: Upper Bound First Estimate}
\end{align}
Now we fix an arbitrary $c>0$ and write
\begin{align}
\notag \frac{1}{(st)^{d/2}}&\int_{(\R^d)^2}\E[\mathrm{e}^{\mathcal{A}_{t,s}(x,y)+\mathcal{B}^d_{t,s}(x,y)}\mathcal{C}_{t,s}(x,y)e^{\mathcal C_{t,s}(x,y)}]dxdy\\&= \label{eqn: Upper Bound Fast Decay}\frac{1}{(st)^{d/2}}\left(\int_{(\R^d)^2}\E[\mathrm{e}^{\mathcal{A}_{t,s}(x,y)+\mathcal{B}^d_{t,s}(x,y)}\mathcal{C}_{t,s}(x,y)e^{\mathcal C_{t,s}(x,y)}\chi_{\{\mathcal{C}_{t,s}(x,y)>c\}}]dxdy\right.
\\&+\left.\label{eqn: Upper Bound Main Term}\int_{(\R^d)^2}\E[\mathrm{e}^{\mathcal{A}_{t,s}(x,y)+\mathcal{B}^d_{t,s}(x,y)}\mathcal{C}_{t,s}(x,y)e^{\mathcal C_{t,s}(x,y)}\chi_{\{\mathcal{C}_{t,s}(x,y)\leq c\}}]dxdy\right)
\end{align}

Beginning with \eqref{eqn: Upper Bound Main Term}, using \eqref{Equation: V Lower Bound} to bound $\mathcal A_{t,s}(x,y)$, we have 
\begin{align*}
&\eqref{eqn: Upper Bound Main Term}\leq\notag\frac1{(st)^{d/2}}\int_{(\R^d)^2}\E[\mathrm{e}^{\mathcal{A}_{t,s}(x,y)+\mathcal{B}^d_{t,s}(x,y)}\mathcal{C}_{t,s}(x,y)e^{\mathcal C_{t,s}(x,y)}\chi_{\{\mathcal{C}_{t,s}(x,y)\}\leq c\}}]dxdy \\
&\leq \frac{e^c}{(st)^{d/2}}\int_{(\R^d)^2}\E[\mathrm{e}^{\mathcal{A}_{t,s}(x,y)+\mathcal{B}^d_{t,s}(x,y)}\mathcal{C}_{t,s}(x,y)]dxdy \notag\\
&\leq
\frac{e^{c}}{(st)^{d/2}}\int_{(\R^d)^2}\E[\mathrm{e}^{-\int_0^ta|x+B_t^{0,0}(u)|^\kappa-b\: du-\int_0^sa|y+B_s^{0,0}(u)|^\kappa-b\: du}e^{\mathcal{B}^d_{t,s}(x,y)}\mathcal{C}_{t,s}(x,y)]dxdy \notag\\
&=
\frac{e^{c+b(s+t)}}{(st)^{d/2}}\int_{(\R^d)^2}\E[\mathrm{e}^{-\frac{a}{t}\int_0^t|t^{1/\kappa}x+t^{1/\kappa}B_t^{0,0}(u)|^\kappa du-\frac{a}{s}\int_0^s|s^{1/\kappa}y+s^{1/\kappa}B_s^{0,0}(u)|^\kappa du}e^{\mathcal{B}^d_{t,s}(x,y)}\mathcal{C}_{t,s}(x,y)]dxdy.
\end{align*}
Applying \eqref{eqn: min triangle inequality} to $|t^{1/\kappa}x+t^{1/\kappa}B_t^{0,0}(u)|^\kappa$ and $|s^{1/\kappa}y+s^{1/\kappa}B_s^{0,0}(u)|^\kappa$ then gives us

\begin{multline*}
\eqref{eqn: Upper Bound Main Term}\leq \frac{e^{C+b(s+t)}}{(st)^{d/2}}\\
\times\int_{(\R^d)^2}\E[\mathrm{e}^{-\frac{a}{t}\int_0^t|t^{1/\kappa}x|^{\kappa\wedge1}-|t^{1/\kappa}B_t^{0,0}(u)|^{\kappa\wedge1}du-\frac{a}{s}\int_0^s|s^{1/\kappa}y|^{\kappa\wedge1}-|s^{1/\kappa}B_s^{0,0}(u)|^{\kappa\wedge 1}du}e^{\mathcal{B}^d_{t,s}(x,y)}\mathcal{C}_{t,s}(x,y)].
\end{multline*}
where $C=c+2a$.
If we denote
\begin{align}
\label{eqn: Brownian Bridge Suprema}
M_t:=\sup_{0\leq u\leq t}|B_t^{0,0}(u)|
\qquad\text{and}\qquad
\overline M_s:=\sup_{0\leq u\leq s}|\overline B_s^{0,0}(u)|
\end{align}
and use the translation invariance of the self-intersection local time, whereby $\mathcal{B}^{d}_{t,s}(x,y)=\mathcal{B}^{d}_{t,s}(0,0)$, then this becomes
\begin{align*}
\eqref{eqn: Upper Bound Main Term} \leq \frac{e^{C+b(s+t)}}{(st)^{d/2}}\E[e^{a(t^{1/\kappa}M_t)^{\kappa\wedge1}+a(s^{1/\kappa}\overline M_s)^{\kappa\wedge1}}e^{{\mathcal{B}^{d}_{t,s}(0,0)}}\int_{(\R^d)^2}e^{-a(|t^{1/\kappa}x|^{\kappa\wedge1}+|s^{1/\kappa}y|^{\kappa\wedge1})}\mathcal{C}_{t,s}(x,y)dxdy].
\end{align*}
Finally, apply H\"older's inequality in order to get
\begin{align}\label{eqn: Upper Bound Main Term Post Holder}
\eqref{eqn: Upper Bound Main Term}\leq \frac{e^{C+b(s+t)}}{(st)^{d/2}}\E&[e^{4a(t^{1/\kappa}M_t)^{\kappa\wedge1}+4a(s^{1/\kappa}\overline M_s)^{\kappa\wedge1}}]^{1/4}\E[e^{4\mathcal{B}^{d}_{t,s}(0,0)}]^{1/4}\notag\\&\cdot\E\left[\left(\int_{(\R^d)^2}e^{-a(|t^{1/\kappa}x|^{\kappa\wedge1}+|s^{1/\kappa}y|^{\kappa\wedge1})}\mathcal{C}_{t,s}(x,y)dxdy\right)^2\right]^{1/2}
\end{align}
Now we consider each of the three expectations in \eqref{eqn: Upper Bound Main Term Post Holder} separately.

First, from \cite[Remark 3.1]{BesselBridgeExponentialMoments}, we have that $M_t$ has finite exponential moments of every order. Therefore, it follows by Brownian scaling that for all $\theta>0$, if $\max\{s,t\}\leq\theta$ then 
\begin{align}\label{eqn: Upper Bound Final Product First Term}
\E[e^{4a(t^{1/\kappa}M_t)^{\kappa\wedge1}+4a(s^{1/\kappa}\overline M_s)^{\kappa\wedge1}}]^{1/4}\leq C_1<\infty
\end{align}
where $C_1$ depends only on $\theta$.
Then applying Propositions \ref{Prop: 1D SILT Bound} and \ref{Prop: 2D SILT Bound}, we get that for $d\in\{1,2\}$ and $\theta>0$ small enough to satisfy Proposition \ref{Prop: 2D SILT Bound}, if $\max\{s,t\}\leq\theta$, then
\begin{align}\label{eqn: Upper Bound Final Product Second Term}
\E[e^{4\mathcal{B}^{d}_{t,s}(0,0)}]^{1/4}\leq C_2<\infty
\end{align}
with $C_2$ again only depending on $\theta$.
Finally, applying Lemma \ref{lem: t+s lower bound} with $\kappa$ in the lemma replaced by $\kappa\wedge1$ and $\eta = \frac{\kappa\wedge1}{\kappa}$ gives us
\begin{align}\label{eqn: Upper Bound Final Product Third Term}
\E\left[\left(\int_{(\R^d)^2}e^{-a|t^{1/\kappa}x|^{\kappa\wedge1}-a|s^{1/\kappa}y|^{\kappa\wedge1}}\mathcal{C}_{t,s}(x,y)dxdy\right)^2\right]^{1/2}\asymp ts(t+s)^{-d/\kappa}
\end{align}
Given that $e^{C+b(s+t)}\asymp1$ and $(s+t)^{-d/\kappa}\asymp\max\{s,t\}^{-d/\kappa}$ as $(s,t)\to0$, we therefore conclude from \eqref{eqn: Upper Bound Main Term Post Holder} that
\begin{align}
\label{eqn: Upper Bound Main Term Proof}
\eqref{eqn: Upper Bound Main Term}\lesssim\min\{s,t\}^{1-d/2}\max\{s,t\}^{1-d/2-d/\kappa}
\qquad\text{as }(s,t)\to0.
\end{align}
It now only remains to show that
\eqref{eqn: Upper Bound Fast Decay}
is of smaller order than this as $(s,t)\to0$.

Toward this end, by H\"older's inequality,
\begin{multline}\label{eqn: second line}
\eqref{eqn: Upper Bound Fast Decay}\leq\frac{1}{(st)^{d/2}}\sup_{x,y\in\mathbb R^d}\P[\mathcal{C}_{t,s}(x,y)>c]^{1/2}\\
\times\sup_{x,y\in\mathbb R^d}\left(\E[e^{8\mathcal{B}^d_{t,s}(x,y)}]^{1/8}\E[\mathcal{C}_{t,s}(x,y)^8]^{1/8}\E[e^{8\mathcal C_{t,s}(x,y)}]^{1/8}\right)\int_{(\R^d)^2}\E[\mathrm{e}^{8\mathcal{A}_{t,s}(x,y)}]^{1/8}dxdy.
\end{multline}
In order to control this quantity, we introduce a H\"older-type inequality for the MILT that will be used repeatedly in the upcoming arguments:

For $d=2$ and $s,t>0$, we claim that for every integer $p\geq1$,
\begin{align}\label{eqn: MILT Holder}
    \E\big[\alpha_{t,s}(B_t^{x,x},\overline B_s^{y,y})^p\big]\leq\E\big[\alpha_{t,t}(B_t^{x,x},\overline B_t^{x,x})^p\big]^{1/2}\E\big[\alpha_{s,s}(B_s^{y,y},\overline B_s^{y,y})^p\big]^{1/2}
\end{align}
We will leave the proof of this inequality to the end of this section.

Taking this for granted for now, we begin by showing that the supremum in the second line in \eqref{eqn: second line} is bounded as $(s,t)\to0$. First, by translation invariance of the SILT and the Brownian scaling coupling in Definition \ref{def: BB Couplings}, we have
\begin{align}
    \notag{\mathcal{B}^d_{t,s}(x,y)}=\mathcal{B}^d_{t,s}(0,0)&= 
\begin{cases}\frac{1}{2}(\beta_{t}(B_t^{0,0})+\beta_{s}(\overline{B}_s^{0,0})) & d=1\\
\gamma_{t}(B_t^{0,0})+\gamma_{s}(\overline{B}_s^{0,0}) & d=2
\end{cases}\\&\overset{d}{=}
\label{eqn: Bd scaling}
\begin{cases}\frac{1}{2}(t^{3/2}\beta_{1}(B_1^{0,0})+s^{3/2}\beta_{1}(\overline{B}_1^{0,0})) & d=1\\
t\gamma_{1}(B_1^{0,0})+s\gamma_{1}(\overline{B}_1^{0,0}) & d=2
\end{cases}
\end{align}
This immediately implies that $\displaystyle\sup_{x,y\in\mathbb R^d}\E[e^{8\mathcal{B}^d_{t,s}(x,y)}]^{1/8}=\E[e^{8\mathcal{B}^d_{t,s}(0,0)}]^{1/8}$ and 
\[\displaystyle\lim_{(s,t)\to0}e^{8\mathcal{B}^d_{t,s}(0,0)}=1\qquad\text{almost surely}\] pointwise in $x$ and $y$. Furthermore, $e^{8\mathcal{B}^d_{t,s}(0,0)}$ is uniformly integrable in $s,t$ that are sufficiently small by a routine computation using the de-la-Vall\'ee-Poussin criterion for uniform integrability combined with Propositions \ref{Prop: 1D SILT Bound} and \ref{Prop: 2D SILT Bound}. Thus, we can apply the Vitali convergence theorem yielding 
\begin{align}\label{eqn: exp SILT limit}
\lim_{(s,t)\to0}\E[e^{8\mathcal{B}^d_{t,s}(0,0)}]^{1/8} = 1.
\end{align}

Now for $d=2$, using \eqref{eqn: MILT Holder} and the MILT scaling from \cite[Lemma 5.14]{GLP} gives us
\begin{align}\label{eqn: 2D MILT Moment scaling}
    \notag \E\big[\alpha_{t,s}(B_{t}^{x,x},\overline B_s^{y,y})^8\big]
    &\leq\E\big[(t\alpha_{1,1}(B_1^{0,0},\overline B_1^{0,0}))^8\big]^{1/2}\E\big[(s\alpha_{1,1}(B_1^{0,0},\overline B_1^{0,0}))^8\big]^{1/2}
    \notag
    \\&=(st)^{4}\E\big[\alpha_{1,1}(B_1^{0,0},\overline B_1^{0,0})^8\big]^{1/2}\E\big[\alpha_{1,1}(B_1^{0,0},\overline B_1^{0,0})^8\big]^{1/2}.
\end{align}
Both expectations are finite by \cite[Lemma 5.14]{GLP} and so
\begin{align}\label{eqn: 2d limit milt moment}
    \lim_{(s,t)\to0} \sup_{x,y\in\mathbb R^d}\E\big[\alpha_{t,s}(B_{t}^{x,x},\overline B_s^{y,y})^8\big] = 0
\end{align} for $d=2$.

When $d=1$, using \eqref{eqn: 1D MILT quadratic form inequality}, SILT translation invariance, and Brownian scaling gives us
\begin{align}\label{eqn: 1D MILT to SILT Bound}
    \E\big[\alpha_{t,s}(B_{t}^{x,x},\overline B_s^{y,y})^8\big]\leq\E\left[\left(\frac{\beta_{t}(B_t^{x,x})+\beta_s(\overline{B}_s^{y,y})}2\right)^8\right]
    =\E\left[\left(\frac{t^{3/2}\beta_{1}(B_1^{0,0})+s^{3/2}\beta_1(\overline{B}_1^{0,0})}2\right)^8\right]
\end{align}
Thus, by another application of dominated convergence we get
\[\lim_{(s,t)\to0} \sup_{x,y\in\mathbb R}\E\big[\alpha_{t,s}(B_{t}^{x,x},\overline B_s^{y,y})^8\big] = 0.\]

Combining \eqref{eqn: 2d limit milt moment}
with \eqref{eqn: 1D MILT to SILT Bound}, we are able to conclude that for $d=1,2$,
\begin{align}\label{eqn: MILT Moment limit}
    \lim_{(s,t)\to0}\sup_{x,y\in\mathbb R^d}\E\big[\mathcal{C}_{t,s}(x,y)^8\big] = 0.
\end{align}

Finally, using the definition of $\mathcal{C}_{t,s}(x,y)$ gives us
\begin{align}
    \E[e^{8\mathcal C_{t,s}(x,y)}] = \sum_{m=0}^\infty\frac{8^m\E\big[\alpha_{t,s}(B_t^{x,x},\overline{B}_s^{y,y})^m\big]}{m!}
\end{align}
When $d=2$, we can bound $\E\big[\alpha_{t,s}(B_t^{x,x},\overline{B}_s^{y,y})^m\big]$ uniformly in $x$ and $y$ using the inequality \eqref{eqn: MILT Holder} and Brownian scaling identically to what we did in order to get \eqref{eqn: 2D MILT Moment scaling}. 
When $d=1$, we argue similarly using \eqref{eqn: 1D MILT quadratic form inequality} and Brownian scaling to bound $\E\big[\alpha_{t,s}(B_t^{x,x},\overline{B}_s^{y,y})^m\big]$ uniformly in $x$ and $y$ in the same manner as we did to get \eqref{eqn: 1D MILT to SILT Bound}. 
In addition, we have that $\E[e^{8\mathcal C_{t,s}(x,y)}]$ is uniformly integrable in $s,t$ sufficiently small by Proposition \ref{Prop: MILT UI} and the de-la-Vall\'ee-Poussin criterion for uniform integrability.
Thus we are able to conclude that 
\begin{align}\label{eqn: exp MILT limit}
\lim_{(s,t)\to0}\sup_{x,y\in\mathbb R^d}\E[e^{8\mathcal C_{t,s}(x,y)}] = 1.
\end{align}

In summary, we can combine \eqref{eqn: exp SILT limit}, \eqref{eqn: MILT Moment limit}, and \eqref{eqn: exp MILT limit} to conclude that
\begin{align}\label{eqn: sup part limit}
    \lim_{(s,t)\to0}\sup_{x,y\in\mathbb R^d}\left(\E[e^{8\mathcal{B}^d_{t,s}(x,y)}]^{1/8}\E[\mathcal{C}_{t,s}(x,y)^8]^{1/8}\E[e^{8\mathcal C_{t,s}(x,y)}]^{1/8}\right)=0
\end{align}

We finish the analysis of the second line, \eqref{eqn: second line}, by using \eqref{eqn: V uniform bound} to get
\begin{multline}\label{eqn: exp MILT Moment Rate}
   \E[\mathrm{e}^{8\mathcal{A}_{t,s}(x,y)}]^{1/8}\leq e^{(a+b)(s+t)}e^{-at|x|^{\kappa\land1}-as|y|^{\kappa\land1}}\\\times\mathbb E\bigg[\exp\bigg(8a\max_{u\in[0,t]}|B_t^{0,0}(u)|^{\kappa\land1}t+8a\max_{u\in[0,s]}|\overline B_s^{0,0}(u)|^{\kappa\land1}s\bigg)\bigg]^{1/8}.
\end{multline}
Since Brownian bridge suprema have finite exponential moments (\cite[Remark 3.1]{BesselBridgeExponentialMoments}),
we conclude that 
\begin{align*}
    \int_{(\mathbb R^d)^2}\E[\mathrm{e}^{8\mathcal{A}_{t,s}(x,y)}]^{1/8}dxdy\lesssim\int_{(\mathbb R^d)^2}e^{-t|x|^{\kappa\land1}-s|y|^{\kappa\land1}}dxdy
    \asymp(ts)^{-d/\kappa\land1}
    \quad\text{as }(s,t)\to0.
\end{align*}
Thus, with this in hand, we may summarize our
analysis of the first line in \eqref{eqn: Upper Bound Fast Decay} as
\begin{align}\label{eqn: second line part 2}
\eqref{eqn: Upper Bound Fast Decay}\lesssim
(ts)^{-d/2-d/\kappa\land1}
\sup_{x,y\in\mathbb R^d}\P[\mathcal{C}_{t,s}(x,y)>c]^{1/2}
\qquad\text{as }(s,t)\to0.
\end{align}
In particular, in order to prove that
\eqref{eqn: Upper Bound Fast Decay} is negligible
compared to \eqref{eqn: Upper Bound Main Term},
it suffices to show that the supremum of the
probabilities in \eqref{eqn: second line part 2} decays faster than any polynomial in $st$.

Toward this end, we note that by definition of $\mathcal C_{t,s}(x,y)$ and Markov's inequality,
\begin{align}\label{eqn: MILT Markov}
    \P\big[\mathcal{C}_{t,s}(x,y)>c] = \P\big[\alpha_{t,s}(B_t^{x,x},\overline B_s^{y,y})>c]\leq\frac{\E\big[\alpha_{t,s}(B_t^{x,x},\overline B_s^{y,y})^p\big]}{c^p}
\end{align}
for every integer $p\geq1$.
In the case $d=2$, we can use \eqref{eqn: MILT Holder} and the scaling we applied in the
first line of \eqref{eqn: 2D MILT Moment scaling}
(with a power of $p$ instead of $8$) to
get
\[\frac{\E\big[\alpha_{t,s}(B_t^{x,x},\overline B_s^{y,y})^p\big]}{c^p}\leq\frac{C_p(ts)^{p/2}}{c^p}\]
for some constant $C_p>0$ that only
depends on $p$.
In the case $d=1$, we can use
\begin{multline}
\label{eqn: 1D MILT Holder}
\alpha_{t,s}(B_t^{x,x},\overline{B}_s^{y,y}) = \int_{-\infty}^\infty  L_t^a(B^{x,x}_t)L_s^a(\overline{B}^{y,y}_s)da\\
\leq\|L_t(B^{x,x}_t)\|_2\|L_s(\overline{B}^{y,y}_s)\|_2
=\beta_{t}(B_t^{x,x})^{1/2}\beta_s(\overline{B}_s^{y,y})^{1/2}
\end{multline}
together with an application of H\"older's inequality,
Brownian scaling, and the invariance of self-intersection local times under translations that
\begin{align*}
    \E\big[\alpha_{t,s}(B_t^{x,x},\overline B_s^{y,y})^p\big]&
    \leq\E\big[\beta_t(B_t^{x,x})^{p}\big]^{1/2}
    \E\big[\beta_s(\overline B_s^{y,y})^{p}\big]^{1/2}
    =(ts)^{3p/4}\E\big[\beta_1(B_1^{0,0})^{p}\big].
\end{align*}
Thus, for $d=1$, we have
\[\frac{\E\big[\alpha_{t,s}(B_t^{x,x},\overline B_s^{y,y})^p\big]}{c^p}\leq\frac{C_p(ts)^{3p/4}}{c^p}\]
for a constant $C_p>0$ that only depends on $p$.

In summary, thanks to \eqref{eqn: MILT Markov}, we get that
for every $q>0$,
\[\sup_{x,y\in\mathbb R^d}\P\big[\mathcal{C}_{t,s}(x,y)>c]^{1/2}\lesssim (st)^q
\qquad\text{as }(s,t)\to0\]
for both $d=1,2$.
Combining this with
\eqref{eqn: Upper Bound First Estimate},
\eqref{eqn: Upper Bound Fast Decay},
\eqref{eqn: Upper Bound Main Term},
\eqref{eqn: Upper Bound Main Term Proof},
and
\eqref{eqn: second line part 2},
this concludes the proof of
\eqref{eq: Main Upper Bound}.

It now remains only to prove the H\"older-type inequality \eqref{eqn: MILT Holder}.
Using the explicit moment formula \cite[Equation (5.17)]{GLP} we have 
\begin{align}\label{eqn: pth moment}
&\notag\E[\alpha_{t,s}(B_t^{x,x},\overline{B}_s^{y,y})^p] =\\
&\notag\frac{1}{p_t(0)p_s(0)}\int_{(\R^2)^p}\left(\left(\sum_{\sigma\in\Sigma_p}\int_{[0,t]_{\leq}^p}p_{u_1}(z_{\sigma(1)}-x)\prod_{j=2}^p p_{u_j-u_{j-1}}(z_{\sigma(j)}-z_{\sigma(j-1)})\cdot p_{t-u_p}(x-z_{\sigma(p)})du\right)\right.\\
&\cdot\left.\left(\sum_{\overline\sigma\in\Sigma_p}\int_{[0,s]_{\leq}^p}p_{u_1}(z_{\overline\sigma(1)}-y)\prod_{j=2}^p p_{u_j-u_{j-1}}(z_{\overline \sigma(j)}-z_{\overline \sigma(j-1)})\cdot p_{s-u_p}(y-z_{\overline \sigma(p)})du)\right)\right)dz
\end{align}
where $\Sigma_p$ is the set of permutations on $\{1,\ldots,p\}$. This formula was proven in \cite{GLP} specifically for $s=t$; this slightly more general statement is immediate from the arguments in that paper. Applying H\"older's inequality in the $dz$ integral in $\eqref{eqn: pth moment}$ implies \eqref{eqn: MILT Holder}.

\subsection{Optimality}\label{optimality proof} 
We now prove the optimality of Theorem \ref{Thm: Main Theorem}; that is, if
\eqref{Equation: V Lower Bound} and
\eqref{Equation: V Upper Bound} both hold, then 
\begin{align}
\label{eqn: Main Lower Bound}C(s,t)\gtrsim\min\{s,t\}^{1-d/2}\max\{s,t\}^{1-d/2-d/\kappa}\qquad\text{as }(s,t)\to0
\end{align}
in addition to \eqref{eq: Main Upper Bound}.
If we combine the elementary inequality $e^x-1\geq x$ for $x\geq0$,
the asymptotic $e^{t\log t/2\pi+s\log s/2\pi}\gtrsim1$ as $(s,t)\to0$,
and the fact that $\mathcal{B}^{d}_{t,s}(x,y)=\mathcal{B}^{d}_{t,s}(0,0)$
by translation invariance of self-intersection
local times, then it follows from
\eqref{eqn: FK Covariance} that 
\[C(s,t)\gtrsim\frac{1}{(st)^{d/2}}\int_{(\R^d)^2}\E[\mathrm{e}^{\mathcal{A}_{t,s}(x,y)+\mathcal{B}^{d}_{t,s}(0,0)}\mathcal{C}_{t,s}(x,y)]dxdy.\]
Next, it follows from \eqref{Equation: V Upper Bound}
that if we define $C_\kappa = c2^{(\kappa-1)_+}$,
then
\[\int_0^tV\big(B^{x,x}_t(s)\big)ds\:
\leq
\int_0^tc|x+B(s)|^\kappa +f\:ds\leq tC_\kappa(|x|^\kappa+M_t^\kappa)+ft,\]
where we recall the definition of $M_t$ in
\eqref{eqn: Brownian Bridge Suprema}. Thus,
since $e^{-f(s+t)}\gtrsim1$ as $(s,t)\to0$,
we get
\begin{align}
\label{eqn: Lower Bound Starting Point}
C(s,t)\gtrsim\frac{1}{(st)^{d/2}}\int_{(\R^d)^2}\E[e^{-C_\kappa t(|x|^\kappa +M_t^\kappa)-C_\kappa s(|y|^\kappa+M_s^\kappa)+\mathcal{B}_{t,s}(0,0)}\mathcal{C}_{t,s}(x,y)]dxdy.
\end{align}

Now it will be convenient for our purposes to split \eqref{eqn: Lower Bound Starting Point} into two pieces as follows:
\begin{align}
C(s,t) \gtrsim &\frac1{(st)^{d/2}}\left(\int_{(\R^d)^2}e^{-C_\kappa(t|x|^\kappa+s|y|^\kappa)}\E[\mathcal{C}_{t,s}(x,y)]dxdy\right.\label{eqn: Lower Bound Main Term} \\
 &\left.+\int_{(\R^d)^2}e^{-C_\kappa(t|x|^\kappa+s|y|^\kappa)}\E[(e^{-C_\kappa(tM_t^\kappa+sM_s^\kappa)+\mathcal{B}^{d}_{t,s}(0,0)}-1)\mathcal{C}_{t,s}(x,y)]dxdy\right).\label{eqn: Lower Bound Fast Decay Term}
\end{align}
Indeed,
Lemma \ref{lem: t+s lower bound} a) implies that 
\begin{align}\label{eqn: Lower Bound Main Term Bound}
\eqref{eqn: Lower Bound Main Term} \asymp (ts)^{1-d/2}(t+s)^{-d/\kappa}\asymp\min\{s,t\}^{1-d/2}\max\{s,t\}^{1-d/2-d/\kappa}\qquad\text{as }(s,t)\to0;
\end{align}
hence \eqref{eqn: Main Lower Bound}
will be proved if we show that 
\eqref{eqn: Lower Bound Fast Decay Term} is of smaller order than this lower bound as $(s,t)\to0$.

Toward this end, by Tonelli's theorem and H\"older's inequality, we have that
\begin{multline}\label{eqn: Lower Bound Decay Term Holder Product}
|\eqref{eqn: Lower Bound Fast Decay Term}|
\leq\E\left[(e^{-C_\kappa(tM_t^\kappa+sM_s^\kappa)+\mathcal{B}^{d}_{t,s}(0,0)}-1)^2\right]^{1/2}\\\times\frac1{(st)^{d/2}}\E\left[\left(\int_{(\R^d)^2}e^{-C_\kappa(t|x|^\kappa+s|y|^\kappa}\mathcal{C}_{t,s}(x,y)dxdy)\right)^2\right]^{1/2}
\end{multline}
By Lemma \ref{lem: t+s lower bound} b),
the second line in \eqref{eqn: Lower Bound Decay Term Holder Product} is asymptotic to $\min\{s,t\}^{1-d/2}\max\{s,t\}^{1-d/2-d/\kappa}$.
Thus, it suffices to show that the first line
of \eqref{eqn: Lower Bound Decay Term Holder Product}
vanishes as $(s,t)\to0$.
By Brownian scaling (in particular, \eqref{eqn: Bd scaling}), it is clear that
\[\lim_{(s,t)\to0}(e^{-C_\kappa(tM_t^\kappa+sM_s^\kappa)+\mathcal{B}^{d}_{t,s}(0,0)}-1)^2=0\qquad\text{almost surely.}\]
Finally, to take the limit outside the expectation
on the first line of \eqref{eqn: Lower Bound Decay Term Holder Product}, we can observe that this expression is uniformly integrable by the de-la-Vall\'ee-Poussin criterion:
\[\displaystyle\sup_{\{s,t:\max(s,t)<\vartheta\}}\E[|e^{-C_\kappa(tM_t^\kappa+sM_s^\kappa)+\mathcal{B}^{d}_{t,s}(0,0)}-1|^3]<\infty\qquad \text{ for } \vartheta>0\ \text{ sufficiently small}.\]
Indeed, this can be easily seen by expanding the cube and using the triangle inequality, applying H\"older's inequality, and then using Propositions \ref{Prop: 1D SILT Bound} and \ref{Prop: 2D SILT Bound}, along with the fact that $M_t\geq0$.
With this, the proof of \eqref{eqn: Main Lower Bound},
and therefore of Theorem \ref{Thm: Main Theorem},
is complete.

\subsection{Proof of Proposition \ref{Prop: Expectation Bound}}\label{expectation bound proof}

By combining \eqref{eqn: FK Expectation},
$\mathcal P^d(t,t)^{1/2}\asymp t^{-d/2}$,
the bounds \eqref{Equation: V Lower Bound} and \eqref{Equation: V Upper Bound},
and $e^{bt},e^{-ft}\asymp1$, we get
\begin{align}\label{eqn: expectation double bound}
    \frac{1}{t^{d/2}}\int_{\R^d}\E\big[e^{-\int_0^tc|x+B_t^{0,0}(u)|^\kappa du + \mathcal{D}^d_t(x)}\big] dx \lesssim\E\big[\mathrm{Tr}[e^{-tH}]\big]\lesssim\frac{1}{t^{d/2}}\int_{\R^d}\E\big[e^{-\int_0^ta|x+B_t^{0,0}(u)|^\kappa  du + \mathcal{D}^d_t(x)}\big] dx.
\end{align}
By applying the change of variables $u\mapsto ut$, the Brownian scaling coupling used in Definition \ref{def: BB Couplings}, and the translation invariance of the self-intersection local time, we get that for $\theta=a,c$,
\begin{align}\label{eqn: matching integral CoV}
\nonumber
&\frac{1}{t^{d/2}}\int_{\R^d}\E[e^{-\int_0^t\theta|x+B_t^{0,0}(u)|^\kappa du + t^{(4-d)/2}\mathcal{D}^d_1(x)}] dx\\
\nonumber
&\hspace{1in}= \frac{1}{t^{d/2}}\int_{\R^d}\E[e^{-\int_0^1\theta|t^{1/\kappa}x+t^{1/\kappa}B_t^{0,0}(ut)|^\kappa  du + t^{(4-d)/2}\mathcal{D}^d_1(0)} ]dx \notag\\
&\hspace{1in}=\frac{1}{t^{d/2}} \int_{\R^d}\E[e^{-\int_0^1\theta|t^{1/\kappa}x+t^{1/\kappa+1/2}B_1^{0,0}(u)|^\kappa  du + t^{(4-d)/2}\mathcal{D}^d_1(0)}] dx.
\end{align}
Now we use the change of variables $x\mapsto t^{-1/\kappa}x$ to get
\begin{align}\label{eqn: matching integral pretriangle}
\eqref{eqn: matching integral CoV} = t^{-d/2-d/\kappa}\int_{\R^d}\E[e^{-\int_0^1\theta|x+t^{1/\kappa+1/2}B_1^{0,0}(u)|^\kappa  du + t^{(4-d)/2}\mathcal{D}^d_1(0)}] dx.
\end{align}
At this point, it suffices to prove that the integral in \eqref{eqn: matching integral pretriangle} converges to a finite positive number as
$t\to0$.

Toward this end, we note that for every $x\in\mathbb R^d$,
\[\lim_{t\to0}e^{-\int_0^1\theta|x+t^{1/\kappa+1/2}B_1^{0,0}(u)|^\kappa  du + t^{(4-d)/2}\mathcal{D}^d_1(0)}=e^{-\theta|x|^\kappa}
\qquad\text{almost surely}.\]
Given that the integral of this limit is finite,
it suffices to show that we can take the $t\to0$ limit outside the expectation and the $dx$ integral.

Choose $0<t_0\leq1$ according to Proposition \ref{Prop: 2D SILT Bound} so that \begin{align}
\label{eqn: Sup e^gamma over t}
\sup_{0<t\leq t_0}\E\big[\exp(2t\gamma_{1}(B_{1}^{0,0}))\big]<\infty.
\end{align}
By \eqref{eqn: min triangle inequality}, we have 
\begin{align*}
\sup_{t\in(0,t_0]}\exp\left({-\int_0^1\theta|x+t^{1/\kappa+1/2}B_1^{0,0}(u)|^\kappa  du}\right) &\leq \exp\left({-\theta\int_0^1|x|^{\kappa\wedge1}-|t_0^{1/\kappa+1/2}M_1|^{\kappa\wedge1}-1du}\right) \\&=
\exp\left({-\theta\left(|x|^{\kappa\wedge1}-|t_0^{1/\kappa+1/2}M_1|^{\kappa\wedge1}-1\right)}\right) 
\end{align*}
where we recall the notation $M_t$
from \eqref{eqn: Brownian Bridge Suprema}.
Then, by H\"older's inequality, 
\begin{multline}\label{eqn: DCT Bound}
\int_{\mathbb R^d}\sup_{t\in(0,t_0]}\E[e^{-\int_0^1\theta|x+t^{1/\kappa+1/2}B_1^{0,0}(u)|^\kappa  du + \mathcal{D}^d_{t}(0)}] dx
\leq 
\E\left[\exp\left({2\theta\left(|M_1|^{\kappa\wedge1}+1\right)}\right)\right]^{1/2}\\
\times
\sup_{t\in(0,t_0]}\E\left[\exp\left(2t^{(4-d)/2}\mathcal{D}_{1}^d(0)\right)\right]^{1/2}\int_{\mathbb R^d}e^{-\theta|x|^{\kappa\land 1}}dx<\infty,
\end{multline}
where we again use \cite[Remark 3.1]{BesselBridgeExponentialMoments} to get that $M_1$ has finite exponential moments,
as well as \eqref{eqn: Sup e^gamma over t}
and Proposition \ref{Prop: 1D SILT Bound} to bound
the exponential moment of $\mathcal D_1^d(0)$.
Hence we can interchange the limit with the $dx$ integral by the dominated convergence theorem.
As for interchanging the limit with the expectation,
by uniform integrability,
it suffices to check that
there exists some $t_0>0$ and $p>1$ such that
\[\sup_{0<t\leq t_0}\E[(e^{-\int_0^1\theta|x+t^{1/\kappa+1/2}B_1^{0,0}(u)|^\kappa  du + t^{(4-d)/2}\mathcal{D}^d_1(0)})^p]<\infty.\] 
This can be done with exactly the same estimates that we have just used for the dominated convergence argument. The proof of Proposition
\ref{Prop: Expectation Bound} is thus complete.

\subsection{Proof of Lemma \ref{lem: t+s lower bound}}\label{Proof of lemma}

Looking at the integral expressions in Lemma \ref{lem: t+s lower bound}, it is clear that it will be helpful to have formulas for the first and second moments of the mutual intersection local time. The following propositions give us convenient representations of those moments in terms of the integrals of certain Gaussian densities.
\begin{prop}\label{Prop: approx Identity Limit}
Fix $d\in\{1,2\}$, $x,y\in\R^d$ and $s,t>0$. Then, 
\begin{align}
    \E[\alpha_{t,s}(B_t^{x,x},\overline{B}_s^{y,y})] = \int_0^t\int_0^s \frac{\exp(-\frac{|x-y|^2}{2z(s,t,u,v)})}{(2\pi z(s,t,u,v))^{d/2}}dudv = \int_0^t\int_0^s p_{z(s,t,u,v)}(x-y)dudv
\end{align}
where 
\[z(s,t,u,v) = \frac{u(t-u)}{t}+\frac{v(s-v)}{s}\]
\end{prop}
\begin{prop}\label{Prop: Product Expectation is Gaussian pdf}
Fix $d\in\{1,2\}$, $x_i,y_i\in\R^d$ for $i=1,2$ and $s,t>0$. Then,
\[\E[\alpha_{t,s}(B_t^{x_1,x_1},\overline{B}_s^{y_1,y_1})\alpha_{t,s}(B_t^{x_2,x_2},\overline{B}_s^{y_2,y_2})]=\int_{[0,t]^2}\int_{[0,s]^2}g^d_{K(s,t,u,v)}(x_1-y_1,x_2-y_2)dudv
\]
where $g^d_{K(s,t,u,v)}$ is a centered multivariate Gaussian density on $\R^d\times\R^d$ with covariance matrix $K(s,t,u,v)=K_0(s,t,u,v)\otimes I_d$
for the matrix function $K_0$
in \eqref{eqn. Bridge difference covariance}.
\end{prop}
We will now finish the proof of Lemma \ref{lem: t+s lower bound} modulo these two propositions which we leave to the end (see Sections \ref{Proof of approx Identity Limit} and \ref{Proof of Product Expectation} respectively).

\subsubsection{Proof of Lemma \ref{lem: t+s lower bound} Part a)}
Using Proposition \ref{Prop: approx Identity Limit} and Tonelli's Theorem we have 

\begin{align}\label{eqn: first moment exponential}
\int_{\R^d}\int_{\R^d}&e^{-C(t^\eta|x|^\kappa+s^\eta|y|^\kappa)}\E[\alpha_{t,s}\notag(B_t^{x,x},\overline{B}_s^{y,y})]dxdy \\&=\notag \int_{\R^d}\int_{\R^d}e^{-C(t^\eta|x|^\kappa+s^\eta|y|^\kappa)}\int_0^t\int_0^sp_{z(s,t,u,v)}(x-y)dudvdxdy
\\&=\int_0^t\int_0^s\int_{\R^d}\int_{\R^d}e^{-C(t^\eta|x|^\kappa+s^\eta|y|^\kappa)}p_{z(s,t,u,v)}(x-y)dudvdxdy.
\end{align}

Now we apply the substitution $w=y-x$, 
\begin{align}\label{eqn: convolution form}
\eqref{eqn: first moment exponential} &= \notag\int_0^t\int_0^s\int_{\R^d}p_{z(s,t,u,v)}(w)\int_{\R^d}e^{-C(t^\eta|x|^\kappa+s^\eta|x+w|^\kappa)}dxdwdudv\\&=\int_0^t\int_0^s\E\left[\int_{\R^d}e^{-C(t^\eta|x|^\kappa+s^\eta|x|+G_{z(s,t,u,v)}|^\kappa)}dx\right]dudv
\end{align}

where $G_{z(s,t,u,v)}$ is a centered $d$-dimensional Gaussian vector with covariance $z(s,t,u,v)I_d$.\\
Letting $T:=t+s$ and using the substitution $x=T^{-\eta/\kappa}x$ yields,
\begin{align}\label{eqn: T rescaling}
\int_{\R^d}e^{-C(t^\eta|x|^\kappa+s^\eta|x+w|^\kappa)}dx =T^{-d\eta/\kappa}\int_{\R^d}e^{-C((\frac{t}{T})^\eta|x|^\kappa+(\frac{s}{T})^\eta|x+T^{\eta/\kappa}w|^\kappa)}dx.
\end{align}
Now we apply the rescalings $a=\frac{u}{t}$ and $b=\frac{v}{s}$ for $a,b\in[0,1]$
to get 
\begin{align}\label{eqn: z rescaling}
z(s,t,u,v)=T(\frac{t}{T}a(1-a)+\frac{s}{T}b(1-b)) =: Tz'(s,t,a,b)
\end{align}
Returning to \eqref{eqn: T rescaling} we apply the above scaling to get
\begin{align*}
\eqref{eqn: first moment exponential}&= 
tsT^{-d\eta/\kappa}\int_0^1\int_0^1\E\left[\int_{\R^d}e^{-C((\frac{t}{T})^\eta|x|^\kappa+(\frac{s}{T})^\eta|x+T^{\eta/\kappa}G_{Tz'(s,t,a,b)}|^\kappa)}dx\right]dadb.
\end{align*}
It now only remains to show that
the integral above can be bounded away from zero and infinity when we take $s$ and $t$ sufficiently close to zero. 

Toward this end, we first prove
boundedness away from infinity:
Since $\frac{t}{T},\frac{s}{T}\in[0,1]$ and $0<\eta\leq1$
(hence
$(\frac{t}{T})^\eta+(\frac{s}{T})^\eta\geq\frac{t+s}{T}=1$), for every $w\in\R^d$,
one has
\begin{align}\label{eqn: min bound}
(\frac{t}{T})^\eta|x|^\kappa+(\frac{s}{T})^\eta|x+w|^\kappa\geq\min\{|x|^\kappa,|x+w|^\kappa\}.
\end{align}
Therefore, we get
\begin{align}
\label{eqn: gaussian conv upper bound}
\int_0^1\int_0^1\E\left[
\int_{\R^d}e^{-C((\frac{t}{T})^\eta|x|^\kappa+(\frac{s}{T})^\eta|x+T^{\eta/\kappa}G_{Tz'(s,t,a,b)}|^\kappa)}dx
\right]dadb\leq 2\int_{\R^d}e^{-C|x|^\kappa}dx<\infty
\end{align}
uniformly over all $s,t>0$.
We now prove boundedness away from zero: Letting $U\subset\R^d$ denote the unit ball, we have
\begin{align}\label{eqn: MILT Convolution Lower Bound}
   \notag\int_{\R^d}e^{-C((\frac{t}{T})^\eta|x|^\kappa+(\frac{s}{T})^\eta|x+T^{\eta/\kappa}w|^\kappa)}dx&\geq\int_{U}e^{-C((\frac{t}{T})^\eta|x|^\kappa+(\frac{s}{T})^\eta|x+T^{\eta/\kappa}w|^\kappa)}dx \notag\\&\geq |U|e^{-C(\frac tT)^\eta}\inf_{|x|\leq1}e^{-C(\frac{s}{T})^\eta|x+T^{\eta/\kappa}w|^\kappa} \notag\\&= |U|e^{-C(\frac tT)^\eta}e^{-C(\frac{s}{T})^\eta(1+T^{\eta/\kappa}|w|)^\kappa}
   \notag\\
   &\geq|U|e^{-C(\frac tT)^\eta}e^{-Cc_\kappa(\frac{s}{T})^\eta(1+T^{\eta}|w|^\kappa)},
\end{align}
where the last inequality follows from $(|x|+|y|)^\kappa\leq c_\kappa(|x|^\kappa+|y|^\kappa)$ for all $x,y\in\mathbb R^d$ with the constant $c_\kappa=2^{(\kappa-1)_+}$.
Since $\frac sT,\frac tT\leq1$, one has
\[|U|e^{-C(\frac tT)^\eta}e^{-Cc_\kappa(\frac{s}{T})^\eta(1+T^{\eta}|w|^\kappa)}
\geq|U|e^{-C-Cc_\kappa}e^{-Cc_\kappa T^{\eta}|w|^\kappa}\]
uniformly in $s,t>0$. Thus, it suffices
to prove that
\begin{align}
\label{eqn: Gaussian open set lower bound}
\lim_{(s,t)\to0}
\int_0^1\int_0^1\E\left[
e^{-Cc_\kappa T^{\eta}|G_{Tz'(s,t,a,b)}|^\kappa}
\right]dadb
>0.
\end{align}
Note that 
$T^{\eta}|G_{Tz'(s,t,a,b)}|^\kappa$ is equal in distribution to $T^{\eta}(Tz'(s,t,a,b))^{\kappa/2}|G_1|^\kappa$,
and that \[T^{\eta}(Tz'(s,t,a,b))^{\kappa/2}\to0\]
pointwise in $a,b$ as $s,t\to0$. Thus,
by dominated convergence
(whereby we can take the limit inside the $dadb$ integral and the expectation because $e^{-x}\leq 1$ for $x\geq0$),
\[\text{LHS of }\eqref{eqn: Gaussian open set lower bound}
=\int_0^1\int_0^1\E\left[\lim_{(s,t)\to0}
e^{-Cc_\kappa T^{\eta}(Tz'(s,t,a,b))^{\kappa/2}|G_1|^\kappa}
\right]dadb
=1.\]
The proof of Lemma \ref{lem: t+s lower bound} part a) is thus complete.

\subsubsection{Proof of Lemma \ref{lem: t+s lower bound} Part b)}
We begin by expanding and applying Proposition \ref{Prop: Product Expectation is Gaussian pdf},
whereby
\begin{multline}\label{eqn: expanded squared exponential milt}
\E\left[\left(\int_{(\R^d)^2}e^{-C(t^\eta|x|^\kappa+s^\eta|y|^\kappa)}\mathcal C_{t,s}(x,y)dxdy\right)^2\right] =\int_{(\R^d)^4}e^{-C(t^\eta|x_1|^\kappa+s^\eta|y_1|^\kappa)}e^{-C(t^\eta|x_2|^\kappa+s^\eta|y_2|^\kappa)}\\\times\int_{[0,t]^2}\int_{[0,s]^2}g^d_{K(s,t,u,v)}(x_1-y_1,x_2-y_2)dudvdxdy.
\end{multline}
Arguing like in part a) we can express this as 
\begin{align}\label{eqn: 2nd moment convolution representation}
&\eqref{eqn: expanded squared exponential milt}\notag = \\&\int_{[0,t]^2}\int_{[0,s]^2}\E\left[\int_{\R^d}e^{-C(t^\eta|x_1|^\kappa+s^\eta|x_1+G^1_{K(s,t,u,v)}|^\kappa)}dx_1\int_{\R^d}e^{-C(t^\eta|x_2|^\kappa+s^\eta|x_2+G^2_{K(s,t,u,v)}|^\kappa)}dx_2\right]dudv,
\end{align}
where $(G^1_{K(s,t,u,v)},G^2_{K(s,t,u,v)})$
is multivariate Gaussian with covariance $K(s,t,u,v)$.
Similarly to part a) we once again apply the rescalings $a_i = \frac{u_i}{t}$ and $b_i = \frac{v_i}{s}$ for $i=1,2$. Using this scaling and equation \eqref{eqn: T rescaling} gives us
\begin{align}\label{eqn: 2nd Moment Convolution final expression}
\eqref{eqn: 2nd moment convolution representation} = t^2s^2T^{-2d\eta/\kappa}\int_{[0,1]^4}\left(\E\left[\int_{\R^d}e^{-C((\frac{t}{T})^\eta|x_1|^\kappa+(\frac{s}{T})^\eta|x_1+T^{\eta/\kappa}G^1_{TK'(s,t,a,b)}|^\kappa)}dx_1\notag\right.\right.\\\left.\left.\cdot\int_{\R^d}e^{-C((\frac{t}{T})^\eta|x_2|^\kappa+(\frac{s}{T})^\eta|x_2+T^{\eta/\kappa}G^2_{TK'(s,t,a,b)}|)^\kappa)}dx_2\right]\right)dadb;
\end{align}
here $K'(s,t,a,b)=K'_0(s,t,a,b)\otimes I_d$,
where
\[K'_0(s,t,a,b)=\frac tTA(a_1,a_2)+\frac sTA(b_1,b_2)\]
for the matrix $A$ defined in \eqref{eqn: A matrix}
(see also the scaling \eqref{eqn: A matrix 2}).

Applying the inequality \eqref{eqn: min bound} gives us 
\[\int_{\R^d}e^{-C((\frac{t}{T})^\eta|x_i|^\kappa+(\frac{s}{T})^\eta|x_i+T^{\eta/\kappa}G^i_{TK'(s,t,a,b)}|^\kappa)}dx_i\leq 2\int_{\R^d}e^{-C|x_i|^\kappa}dx<\infty\] 
uniformly in $s,t>0$ for both $i=1,2$.
On the other hand, we can apply the inequality \eqref{eqn: MILT Convolution Lower Bound} and use the fact that $\frac{s}{T},\frac{t}{T}\leq1$ to get
\[\int_{\R^d}e^{-C((\frac{t}{T})^\eta|x_i|^\kappa+(\frac{s}{T})^\eta|x_i+T^{\eta/\kappa}G^i_{TK'(s,t,a,b)}|^\kappa)}dx_i\geq|U|e^{-C-Cc_\kappa}e^{-Cc_\kappa T^\eta|G^i_{TK'(s,t,a,b)}|^\kappa}\]
uniformly in $s,t>0$ for both $i=1,2$.
Therefore, it suffices to prove that
\begin{align}\label{eqn: 2nd moment final limit}
    \lim_{(s,t)\to0}\int_{[0,1]^4}\E\left[e^{-Cc_\kappa T^\eta|G^1_{TK'(s,t,a,b)}|^\kappa}e^{-Cc_\kappa T^\eta|G^2_{TK'(s,t,a,b)}|^\kappa}\right]dadb>0.
\end{align}
The proof of this
is identical to \eqref{eqn: Gaussian open set lower bound}; this completes the proof of Lemma \ref{lem: t+s lower bound} part b).

\subsection{Proof of Proposition \ref{Prop: approx Identity Limit}}\label{Proof of approx Identity Limit}
Recall that the Brownian Bridge in $\R^d$ is simply defined as a vector with an independent Brownian Bridge in each component. In particular, for $x\in\R^d$, $t>0$, and $u\in[0,t]$ we have,
\begin{align}\label{eqn: 2D Bridge Distribution}
    B_t^{x,x}(u)\sim\mathcal{N}\left(x,\frac{u(t-u)}{t}I_d\right).
\end{align}
By \eqref{eqn: approx milt},
\[\E[\alpha^\epsilon_{t,s}(B_t^{x,x},\overline{B}_s^{y,y})]=\int_0^t\int_0^s\E[p_\epsilon(B_t^{x,x}(u)-\overline{B}_s^{y,y}(v))]dudv.\]

Fix $x,y,u,v,s,t$.
Let $\Phi_{B}=\Phi_{B_t^{x,x}(u)}$ and $\Phi_{\overline{B}}=\Phi_{\overline{B}_s^{y,y}(v)}$ denote the respective densities of $B_t^{x,x}$ and $\overline{B}_s^{y,y}$. Then by the independence of the two bridges we have,
\begin{align*}
\E[p_\epsilon(B_t^{x,x}(u)-\overline{B}_s^{y,y}(v))] = 
\int_{(\R^d)^2}p_\epsilon(z_1-z_2)\Phi_B(z_1)\Phi_{\overline{B}}(z_2)dz_1dz_2.
\end{align*}
Applying Young's Convolution Inequality gives us
\begin{align*}
    \int_{(\R^d)^2}p_\epsilon(z_1-z_2)\Phi_B(z_1)\Phi_{\overline{B}}(z_2)dz_1dz_2
    &\leq\norm{p_\epsilon}_1\norm{\Phi_B}_2\norm{\Phi_{\overline{B}}}_2\\&=C\sqrt{\left(\frac{t}{u(t-u)}\right)^{d/2}}\sqrt{\left(\frac{s}{v(s-v)}\right)^{d/2}},
\end{align*}
where we used the fact that the $L^2$-norm of the $d$-dimensional Gaussian density with covariance matrix $\Sigma$ is given by $\frac{1}{(4\pi)^{d/4}\mathrm{det}(\Sigma)^{1/4}}$

Thus by the dominated convergence theorem we get,
\[\lim_{\epsilon\to0}\int_0^t\int_0^s\E[p_\epsilon(B_t^{x,x}(u)-\overline{B}_s^{y,y}(v))]dudv = \int_0^t\int_0^s\lim_{\epsilon\to0}\E[p_\epsilon(B_t^{x,x}(u)-\overline{B}_s^{y,y}(v))]dudv.\]

Consider now the integrand 
\[\lim_{\epsilon\to0}\E[p_\epsilon(B_t^{x,x}(u)-\overline{B}_s^{y,y}(v))]\]
and recall that we can write

\[\E[p_\epsilon(B_t^{x,x}(u)-\overline{B}_s^{y,y}(v))]=\E[p_\epsilon((x-y)+B_t^{0,0}(u)-\overline{B}_s^{0,0}(v))]\]

Notice that $B_t^{0,0}(u)-\overline{B}_s^{0,0}(v)$ is a centered Gaussian with covariance 
\[\left(\frac{u(t-u)}{t}+\frac{v(s-v)}{s}\right)I_d =: z(s,t,u,v)I_d\]

and so the convolution of Gaussian densities yields
\begin{align*}
\E[p_\epsilon((x-y)+B_t^{0,0}(u)-\overline{B}_s^{0,0}(v))] &= \int_{\R^d}p_\epsilon((x-y)+w)p_{z(s,t,u,v)}(w)dw \\&=
p_{\epsilon+z(s,t,u,v)}(x-y)
\end{align*}

Finally, we can conclude
\begin{align*}
\int_0^t\int_0^s\lim_{\epsilon\to0}\E[p_\epsilon(B_t^{x,x}(u)-\overline{B}_s^{y,y}(v))]dudv&=\int_0^t\int_0^s\lim_{\epsilon\to0}p_{\epsilon+z(s,t,u,v)}(x-y)dudv\\&=
\int_0^t\int_0^sp_{z(s,t,u,v)}(x-y)dudv
\end{align*}
completing the proof.

\subsection{Proof of Proposition \ref{Prop: Product Expectation is Gaussian pdf}}\label{Proof of Product Expectation}

Note that $\mathrm{Cov}\big[B_t^{x,x}(u),B_t^{x,x}(v)\big]=u\wedge v -\frac{uv}{t}$. 
Therefore, the vector of $\R$-valued Brownian Bridge differences
\[(B_t^{0,0}(u_1)-\overline{B}_s^{0,0}(v_1),B_t^{0,0}(u_2)-\overline{B}_s^{0,0}(v_2))\] has covariance matrix

\begin{align}\label{eqn. Bridge difference covariance}
K_0(s,t,u,v) = 
\begin{pmatrix}
\frac{u_1(t-u_1)}{t}+\frac{v_1(s-v_1)}{s}&(u_1\wedge u_2-\frac{u_1u_2}{t})+(v_1\wedge v_2-\frac{v_1v_2}{s})\\
(u_1\wedge u_2-\frac{u_1u_2}{t})+(v_1\wedge v_2-\frac{v_1v_2}{s})&\frac{u_2(t-u_2)}{t}+\frac{v_2(s-v_2)}{s}
\end{pmatrix}
\end{align}
Then, since a Brownian bridge in $\R^d$ has $d$ independent coordinates, the vector of the $\R^d$-valued Brownian Bridge differences
\[(B_t^{0,0}(u_1)-\overline{B}_s^{0,0}(v_1),B_t^{0,0}(u_2)-\overline{B}_s^{0,0}(v_2))\] has covariance matrix $K(s,t,u,v)=K_0(s,t,u,v)\otimes I_d$

Now let $g_{\Sigma}^d(w_1,w_2)$ be the centered Gaussian density on $\R^d\times\R^d$ with covariance matrix $\Sigma$. 
Then \eqref{eqn: approx milt} gives us 
\begin{align}\label{eqn: Expectation of squared approx MILT}
\E[&\alpha^\epsilon_{t,s}(B_t^{x_1,x_1},\overline{B}_s^{y_1,y_1})\alpha^\epsilon_{t,s}(B_t^{x_2,x_2},\overline{B}_s^{y_2,y_2})] \notag\\&= \int_{[0,t]^2}\int_{[0,s]^2}\E\left[p_\epsilon(B_t^{x_1,x_1}(u_1)-\overline{B}_s^{y_1,y_1}(v_1))p_\epsilon(B_t^{x_2,x_2}(u_2)-\overline{B}_s^{y_2,y_2}(v_2))\right] dudv
\notag
\\&=\int_{[0,t]^2}\int_{[0,s]^2}g^d_{K(s,t,u,v)+\epsilon I_{2d}}(x_1-y_1,x_2-y_2)dudv.
\end{align}
Note that the matrix $K(s,t,u,v)$ is positive definite
for almost every $u,v$ (i.e. $\det(K(s,t,u,v))=0$ may only occur when $u_1=u_2$, $v_1=v_2$, $u_i\in\{0,t\}$, or $v_i\in\{0,s\}$).
Outside of that defective null set, $g^d_{K(s,t,u,v)+\epsilon I_{2d}}\to g^d_{K(s,t,u,v)}$ as $\epsilon\to0$. Thus,
Proposition \ref{Prop: Product Expectation is Gaussian pdf} can be reduced to verifying that we can take the limit
outside the $dudv$ integral in \eqref{eqn: Expectation of squared approx MILT}.

Toward this end,
by the dominated convergence theorem, it suffices
to find an integrable function that dominates
$g^d_{K(s,t,\cdot,\cdot)+\epsilon I_{2d}}(x_1-y_1,x_2-y_2)$ almost everywhere for all $\epsilon>0$.
For this, we use the simple bounds
\begin{multline*}
|g^d_{K(s,t,u,v)+\epsilon I_{2d}}(x_1-y_1,x_2-y_2)|
\leq(2\pi)^{-d}\det(K(s,t,u,v)+\epsilon I_{2d})^{-1/2}\\
\leq(2\pi)^{-d}\det(K(s,t,u,v))^{-1/2}
=(2\pi)^{-d}\det(K_0(s,t,u,v))^{-d/2},
\end{multline*}
which holds whenever $K(s,t,u,v)\neq0$, and
where the first inequality on the second
line follows from the fact that adding the matrix $\epsilon I_{2d}$ can only increase the determinant
since it adds $\epsilon$ to every eigenvalue.
Thus, we only need the integrability of $\mathrm{det}(K_0(s,t,\cdot,\cdot))^{-d/2}$.
 
In order to simplify this computation we introduce the scaled matrix 
\begin{align}
\label{eqn: A matrix}
A(a_1,a_2)=\begin{pmatrix}
a_1-a_1^2&a_1\wedge a_2-a_1a_2\\
a_1\wedge a_2-a_1a_2&a_2-a_2^2
\end{pmatrix}
\end{align}
for $a_1,a_2\in[0,1]$.
Notice that under the scalings $u_i=ta_i$ and $v_i=sb_i$, we have
\begin{align}
\label{eqn: A matrix 2}
K_0(s,t,u,v) = tA(a_1,a_2)+sA(b_1,b_2)
\end{align}

A routine computation gives us 
\[\mathrm{det}\big(A(a_1,a_2)\big) =(a_1\wedge a_2)(1-(a_1\vee a_2))|a_1-a_2|\]
and Minkowski's determinant inequality yields
\begin{align}\label{eqn: Minkowski}
\mathrm{det}(K_0(s,t,u,v))\geq t^2\mathrm{det}(A(a_1,a_2))+s^2\mathrm{det}(A(b_1,b_2))
\end{align}
Given the scaling inequality \eqref{eqn: Minkowski} it suffices to show
\[\int_{[0,1]^4}(t^2\mathrm{det}(A(a_1,a_2))+s^2\mathrm{det}(A(b_1,b_2)))^{-d/2}dadb<\infty\]
For this, the AM-GM inequality gives us 
\begin{align*}
\int_{[0,1]^4}(t^2\mathrm{det}A(a_1,a_2)&+s^2\mathrm{det}A(b_1,b_2))^{-d/2}dadb\\&\leq C_{d,t,s}\int_{[0,1]^4}\mathrm{det}(A(a_1,a_2))^{-d/4}\mathrm{det}A((b_1,b_2))^{-d/4}dadb,
\end{align*}
where $C_{d,t,s}>0$ is independent of $u,v$.
Therefore, by Tonelli's theorem, it suffices to show
\[\int_{[0,1]^2}\mathrm{det}(A(a_1,a_2))^{-d/4}da_1da_2<\infty.\]

First we have,
\[\int_{[0,1]^2}(\mathrm{det}A(a_1,a_2))^{-d/4}da_1da_2=2\int_{0<a_1<a_2<1}(a_1(a_2-a_1)(1-a_2))^{-d/4}da_1da_2.\]

Then applying the changes of variables $x=a_1$ and $y=a_2-a_1$ gives us
\[\int_{0<a_1<a_2<1}(a_1(a_2-a_1)(1-a_2))^{-d/4}da_1da_2=\int_{\{x,y>0,x+y<1\}}(xy(1-x-y))^{-d/4}dxdy;\]
this is integrable for $d=1,2$ because $-d/4>-1$
(see, e.g., \cite[(5.14.2)]{NIST:DLMF}).
The proof of Proposition \ref{Prop: Product Expectation is Gaussian pdf} is now complete.

\newpage
\bibliographystyle{alpha}
\bibliography{Bibliography}

 \end{document}